\documentclass[11pt, leqno]{amsart}
\usepackage{amsmath,amssymb,amsthm,enumerate,bbm}
\usepackage{graphicx}
\usepackage{booktabs}
\usepackage[font=small]{caption} 
\usepackage[flushleft]{threeparttable}
\usepackage{mathtools,breqn}
\usepackage{url}
\def\arXiv#1{arXiv:\href{http://arXiv.org/abs/#1}{#1}}

\def\MR#1{MR\href{http://www.ams.org/mathscinet-getitem?mr=#1}{#1}}

\def\spine{1.1in}
\usepackage[
heightrounded,
top=1.15in,
bottom=1.2in,
inner=\spine,
outer=\spine 
]{geometry} % 

\usepackage[colorlinks,linkcolor=blue,citecolor=magenta,urlcolor=black,hypertexnames=false]{hyperref} 

\def\bs#1{\boldsymbol{#1}}
\def\C{\mathbb C}
\def\CP{\mathbb {CP}}
\def\dim{\operatorname{dim}}

\def\F{\mathbb F}
\def\FP{\mathbb {FP}}

\def\N{\mathbb N}
\def\O{\mathbb O}
\def\OP{\mathbb {OP}}
\def\P{\mathcal P}

\def\R{\mathbb R}
\def\RP{\mathbb {RP}}

\def\S{\mathbb S}

\def\supp{\operatorname{supp}}
\def\theta{\vartheta}

\numberwithin{equation}{section}

\newtheorem{theorem}{Theorem}[section]
\newtheorem{lemma}[theorem]{Lemma}

\newtheorem{proposition}[theorem]{Proposition}
\newtheorem{corollary}[theorem]{Corollary}
\newtheorem{definition}[theorem]{Definition}
\theoremstyle{definition}
\newtheorem{example}[theorem]{Example}
\newtheorem{conjecture}[theorem]{Conjecture}

\title[Optimal measures for $p$-frame energies on spheres]{Optimal measures for $\bs p$-frame energies on spheres}
\date{\today}
\author{Dmitriy Bilyk}
\author{Alexey Glazyrin}
\author{Ryan Matzke}
\author{Josiah Park}
\author{Oleksandr Vlasiuk}
\address{School of Mathematics, University of Minnesota, Minneapolis, MN 55455} 
\email{dbilyk@math.umn.edu}
\address{School of Mathematical \& Statistical Sciences, The University of Texas Rio Grande Valley, Brownsville, TX 78520}
\email{alexey.glazyrin@utrgv.edu}
\address{School of Mathematics, University of Minnesota, Minneapolis, MN 55455} 
\email{matzk053@umn.edu}
\address{School of Mathematics, Georgia Institute of Technology, Atlanta, GA 30332}
\email{j.park@gatech.edu}
\address{Department of Mathematics, Florida State University, Tallahassee, FL 32306}
\email{ovlasiuk@fsu.edu}

\begin{document} 

\begin{abstract}
  We provide new answers about the placement of  mass on spheres so as to minimize energies   of pairwise interactions. We find optimal measures for the $p$-frame energies, i.e. energies with the kernel given by the absolute value of the inner product raised to a positive power $p$.  Application of linear programming methods in the setting of projective spaces allows for describing the minimizing measures in full in several cases:
  we show optimality of tight designs and of the $600$-cell for several ranges of $p$ in different dimensions. Our methods apply to a much broader class of potential functions, those which are absolutely monotonic up to a particular order  as functions of the cosine of the geodesic distance.
 In addition, a preliminary numerical study is presented which suggests optimality of several other highly symmetric configurations and weighted designs in low dimensions. In one case we improve the best known lower bounds on a minimal sized weighted design in $\CP^4$.  All these results point to the discreteness of minimizing measures for the $p$-frame energy with $p$ not an even integer.  
\end{abstract} 

\maketitle

\section{Introduction} 

An intriguing natural phenomenon is the ubiquitous appearance of certain symmetric structures and configurations as solutions to optimization problems. In a number of spaces, highly symmetric configurations of points such as the vertices of the icosahedron on $\S^2$ or the minimal vectors of the Leech lattice $\Lambda_{24}$ on $\S^{23}$ are optimal codes \cite{levenshteinDesigns1992}.  First papers on spherical designs made important connections between symmetry and optimality through pioneering work on linear programming bounds \cite{delsarteSpherical1977}. Some highly symmetric configurations, in addition to being spherical designs and optimal codes,  are also minimizers of harmonic energies \cite{Andreev1999, kolushovExtremal1997, kolushov1994, yudinMin1992, yudinLow1997}.

 For a finite configuration of points on the sphere $\mathcal{C}\subset\S^{d-1}$ (also known as a \textit{code}) the discrete  $f$-potential energies are defined as
\begin{equation}\label{eq:poten}
   E_f (\mathcal C) = \frac{1}{|\mathcal C|^2}  \sum\limits_{x, y\in\mathcal{C}} f( \langle x, y \rangle).
\end{equation} 
(The  diagonal terms should be excluded if the kernel $f$ is singular  at $1$, i.e. when $x=y$.) {\it{Universally optimal}} point configurations,  i.e. collections of points $\mathcal C$ minimizing the discrete energies $E_f$ among all point sets of fixed cardinality $|\mathcal C|$,  for all absolutely  monotonic functions $f$ on $[-1,1)$, have been discovered through the linear programming approach of Cohn and Kumar in \cite{cohnUni2007}.

In contrast to the above setting, in the present  paper, rather than considering configurations of fixed cardinality, we focus on the problem of minimizing energies over {\it{all Borel  probability measures}}, discovering that surprisingly in many situations the minimizing measures are discrete.  For a kernel function $f \in  C[-1,1]$ and a Borel  measure $\mu$ on $\S^{d-1}$, we define the energy integral as
\begin{equation}\label{eq:muener}
    I_{f}(\mu)=\int_{\S^{d-1}}\int_{\S^{d-1}} f (\langle x,y \rangle )  d\mu(x)d\mu(y).
\end{equation} 

One is naturally interested in minimizing these energies over $\mu \in \mathcal P (\S^{d-1})$, the set of all Borel probability measures on $\S^{d-1}$, i.e. finding the equilibrium distribution of unit mass under the interaction given by the potential function $f$. This definition is compatible with the discrete energy \eqref{eq:poten} in the sense that 
\begin{equation}\label{eq:trans}
E_f (\mathcal C) =  I_f \Big( \frac{1}{|\mathcal C|}  \sum_{x\in \mathcal C } \delta_x \Big),
\end{equation}
and we shall repeatedly abuse the notation when saying that a configuration $\mathcal C$ minimizes the energy $I_f (\mu)$, to mean that the corresponding measure in the right hand side of the above equation minimizes. 

While many classical examples, such as the Riesz energy,  feature increasing  kernels $f$ which give rise to energies with repulsive interactions (i.e. $f$ is largest when $x=y$ and smallest when $x$ and $y$ are antipodal), we will concentrate on the {\it attractive-repulsive} potentials, which decrease at first,  but increase eventually,   as functions of geodesic distance: in other words, a pair of points will repel when close together, but attract when far apart. Such potentials in $\R^d$  appear naturally for self-assembly models in computational chemistry, emerging collective behavior in population biology,  and  in many other scientific models \cite{wu2015nonlocal,balague2013nonlocal,carrillo2017geometry,von2012predicting,kolokolnikov2011stability,carrillo2003kinetic,mogilner2003mutual}.

We will mostly  consider attractive-repulsive potentials on the sphere  which are symmetric and {\it orthogonalizing}, so that $f(t) = f (|t|)$, $f (t)$ is increasing for $t\in [0,1]$, and  $f$ takes its minimal value at zero. For such potentials, the discrete energy for up to $d$ particles is minimized by collections of orthogonal vectors. Since in this setting the energy does not change by replacing any $x$ with $\lambda x$, where $|\lambda | =1$, its analysis naturally lends itself to the projective space $\mathbb R \mathbb P^{d-1}$, where the potential becomes repulsive, and we adopt this approach in the technical parts of the paper. 

The main examples of the above potentials, which motivate the current paper, are of the form $f(t) = |t|^p$, $p>0$,  which yield 
the {\it $p$-frame energies}: 
\begin{equation}\label{eq:pframe}
    I_{f}(\mu)=\int_{\S_{\F}^{d-1}}\int_{\S_{\F}^{d-1}} |\langle x,y \rangle|^{p} d\mu(x)d\mu(y), 
\end{equation} 
 where $\S_{\F}^{d-1}=\{x\in\F^d\ |\ \|x\|=1\}.$ For $\F=\R$ or $\C$ this type of  energy has a rich history. 
 
 When $p=2$ and $\F=\R$, the discrete version of this energy, known simply as the {\em{frame energy}} or {\em{frame potential}}, has been introduced by Benedetto and Fickus \cite{benedetto2003finite}:  they showed that global (as well as local) minimizers of this energy are precisely unit norm {\it tight frames}. These configurations, which explain the nomenclature  ``frame energy'',  play an important role in signal processing and other branches of applied mathematics and  behave like overcomplete orthonormal bases. A finite collection of vectors $\mathcal C\subset \F^{d}$ is a tight frame, if for any $x \in \F^d$, and some constant $A>0$, one has an analog of Parseval's identity holding for $\mathcal C$,
 \begin{equation}
\sum_{y\in \mathcal C }\  | \langle x,y \rangle|^2\ =A\| x\|^2. 
 \end{equation} These objects also minimize the continuous energy $I_f (\mu)$ for $p=2$, but there are also other minimizers, such as the surface area, or Haar measure $\sigma$ on $\S_{\F}^{d-1}$, and, more generally, {\it isotropic probability measures} on the sphere, i.e. those measures for which $$\displaystyle{\int_{\S_{\F}^{d-1}} |\langle x,y \rangle|^{2} d\mu (y)  = \frac1{d}}, $$ holds for all $x\in {\S_{\F}^{d-1}}$.

 When $p=4$, this energy plays an important role in connection to complex maximal {\it equiangular tight frames}, also known as {\it symmetric, informationally complete, positive operator-valued measures} (SIC-POVMs), i.e. unit norm tight frames in $\C^d$ which satisfy $| \langle x,y \rangle | = \operatorname{const}$ for $x\neq y \in \mathcal C$ and $|\mathcal C| = d^2$ \cite{RBSC}. The existence of these objects is the subject of {\em Zauner's conjecture}, and much of the numerical evidence for this conjecture comes from the observation that they minimize the $4$-frame energy among other energies, as projective $2$-designs, see e.g. \cite{SG} (in what follows, we demonstrate that they also minimize the $p$-frame energy for $2\leq p \leq 4$). 

 More generally, for even integers $p$, these energies were considered in earlier works \cite{sidel1974new,Welch1974,venkovRes2001}, and it is known that for $\F=\R$ or $\C$ projective $k$-designs are precisely the finite configurations which minimize the $p=2k$ energy. In this terminology, unit norm tight frames are equivalent to projective $1$-designs (see  Section \ref{s:des} for precise definitions), while spherical $2$-designs are exactly those  unit norm tight frames, whose center of mass is at the origin. These were constructively shown to exist for $d\geq 2$ precisely when the number of points $N$ satisfies $N\geq d+1$ and $N\neq d+2$ when $d$ is odd \cite{mimura}. The last restriction does not apply to unit norm tight frames, and these exist for all $N\geq d$ \cite{benedetto2003finite}. Surface measure is also known to be a minimizer for $p \in 2\mathbb N$:  this can be seen either from the definition of $k$-designs, or from the fact that the function $f$ is positive definite in this case (see Proposition \ref{prop:pdmin}), and was originally  proved in the real case in \cite{sidel1974new}.

For $p$ not an even integer, optimal distributions of mass for $p$-frame energies are much less studied, to the point of there only being one result on these minimizing measures readily found in the literature. It states  that distributing mass equally on the orthoplex or cross-polytope, an orthonormal basis and its antipodes, gives the unique symmetric minimizer, up to orthogonal transformations, for any energy with $p\in(0,2)$ \cite{Ehler2012}.

This result (contained in our Theorem \ref{t:mainsphere} below as a special case) points to an interesting distinction. When $p$ is even, the $p$-frame energy has a multitude of both continuous, e.g. $\sigma$, and discrete minimizers. However, this is not the case when $p$ is not an even integer: $\sigma$ is no longer a minimizer, since the function $f(t) = |t|^p$ is not positive definite, and so the above result, along with our numerical studies, points to existence of discrete minimizers only.  

In this paper we give a first description of minimizers for several dimensions and some ranges of $p$.  The description relies on the notion of {\it{tight designs}}: designs of high strength, but with few distinct pairwise distances, see Definition \ref{d:tight}. We show that if  there exists a tight projective $M$-design (which in the real case  is equivalent to a tight spherical $(2M+1)$-design), then it minimizes the $p$-frame energy for $p\in(2M-2,2M)$. The $600$-cell, despite  not being a tight design, minimizes the $p$-frame energy for $p\in(8,10)$ among probability measures on $\S^3$, as we show in Section \ref{sec:600cell}.

\begin{theorem}\label{t:mainsphere}
Let $f(t) = |t|^p$, $t\in [-1,1]$. 
\begin{enumerate}[(i)]
\item\label{t1} If there exists a tight spherical $(2M+1)$-design $\mathcal{C} \subset \mathbb{S}^{d-1}$, then    the measure
$$ \mu = \frac{1}{| \mathcal{C}|} \sum_{x \in \mathcal{C}} \delta_{x}$$
is a minimizer of the $p$-frame energy $I_{f}$ with $2M-2 \leq p \leq 2M$   over $\mu \in \mathcal P (\S^{d-1})$.
\item\label{t2} Let $\F = \R$, $\C$ or $\mathbb{H}$.  Assume that there exists a tight projective $M$-design $\widetilde{\mathcal{C}} \subset {\F\mathbb{P}}^{d-1}$, and let the  code $\mathcal C \subset \mathbb{S}_{\F}^{d-1}$ consist of the representers of $\widetilde{\mathcal C}$ in $\mathbb{S}_{\F}^{d-1}$ according to \eqref{eq:linemodel}. Then    the measure
$$ \mu = \frac{1}{| \mathcal{C}|} \sum_{x \in \mathcal{C}} \delta_{x}$$
is a minimizer of the $p$-frame energy $I_{f}$  %on $ \mathbb{S}_{\F}^{d-1}$ 
 with $2M-2 \leq p \leq 2M$  over $\mu \in \mathcal P (\S^{d-1}_\F)$.
\item\label{t3} Let $\mathcal C \subset \S^{3}$ denote the  $600$-cell. Then the measure $$ \mu = \frac{1}{| \mathcal{C}|} \sum_{x \in \mathcal{C}} \delta_{x}$$
is a minimizer of the $p$-frame energy $I_{f}$ with $8 \leq p \leq 10$  over $\mu \in \mathcal P (\S^{3})$.
\end{enumerate}
\end{theorem}

 For  parts \eqref{t1}-\eqref{t2} of the above theorem we also prove a   uniqueness statement: more precisely, whenever the corresponding statements hold, and additionally $p$ is not an endpoint of the interval, i.e. $p\in (2M-2,2M)$,  {\em all} minimizers have to be tight designs (although not necessarily coinciding with $\mathcal C$), in particular, they have to be discrete. Since tight $(2M+1)$-designs on the circle consist just of $2(M+1)$ equally spaced points, the above result fully characterizes the minimizers for $d=2$ (for both the sphere and real projective space). See Section \ref{sec:uniq} for more details. 

We observe that part  \eqref{t1} is essentially contained in   part \eqref{t2} with $\F = \R$: indeed, odd-strength tight spherical designs are necessarily symmetric, and by taking one point in each antipodal pair one obtains a tight projective design (see Sections \ref{s:des}--\ref{subsec:antip} for a more extensive discussion).

Minimizing  the continuous energy \eqref{eq:pframe} over all {\em{measures}} and obtaining discrete minimizers allows us to make new conclusions about the minimizing configurations of the discrete energies \eqref{eq:poten} for certain values of the cardinality $N$.  One directly obtains the following corollary:
\begin{corollary} Let $\F$, $d$, $p$, and $\mathcal C$ be as in any of the parts of Theorem \ref{t:mainsphere}, and let $N = k |\mathcal C |$, $k \in \mathbb N$. Then  $N$-point discrete $p$-frame energy is minimized by the configuration $\mathcal C$ repeated $k$ times, i.e.  
\begin{equation}
\min_{\substack{\mathcal C' \subset \S^{d-1}_\F \\  |\mathcal C'| = N }} \frac{1}{N^2}  \sum\limits_{x, y\in\mathcal{C'}}  \big|  \langle x, y \rangle \big|^p  = I_{|t|^p}  \Big( \frac{1}{|\mathcal C|}  \sum_{x\in \mathcal C } \delta_x \Big). 
\end{equation}
\end{corollary}
\noindent Thus, for example, if $N$ is a multiple of $6$, then repeated copies of a ``half'' of the icosahedron minimize the $N$-point  $p$-frame energy on $\S^2$ for $p \in [2,4]$. Some other results about the minima of discrete $p$-frame energies have been obtained  in  \cite{CGGKO}. \\

The arguments proving   Theorem \ref{t:mainsphere} are strongly reminiscent of those appearing in \cite{cohnUni2007} and are based on the linear programming method which goes back to  Delsarte and Yudin \cite{delsarte1973algebraic,yudinMin1992}. Part \eqref{t2}  of Theorem~\ref{t:mainsphere} is   a consequence of a much more general Theorem \ref{thm:tight}. The latter theorem, in fact, demonstrates that tight $M$-designs possess a certain {\em{universality}} property: they minimize the energy {\em{for all}} strictly monotonic functions of degree exactly $M$ over {\em{all probability measures}}, see Section~\ref{sec:tight} for details.

The proof of optimality for the $600$-cell is computer assisted 
and makes use of the fact that the averages of spherical harmonics over the $600$-cell vanish for a few orders above its maximal strength as a spherical design -- the same idea was used in the proof of universal optimality of the $600$-cell in \cite{cohnUni2007}, as well as earlier in \cite{Andreev1999, Andreev2000}. This allows us to construct a collection of interpolating polynomials $h$ for each $p$ which have the desired properties of lying below $f$, agreeing with $f$ on the distances appearing in $\mathcal{C}$, and finally being positive definite, the last of which is checked using interval arithmetic. The details of the proof are taken up in Section~\ref{sec:600cell}.

We collect all the  necessary preliminary material in  Section \ref{sec:prelim}: Section~\ref{subsec:2point} contains the discussion of relevant properties of compact 2-point homogeneous connected spaces; Section~\ref{sec:energies} explains the specifics of minimizing energy functionals over probability measures on such spaces; Section~\ref{s:des} introduces designs, and, in particular, tight designs; and Section~\ref{subsec:antip} describes the transference between energies on projective spaces and spheres, which connects Theorem \ref{thm:tight} to Theorem \ref{t:mainsphere}.

Extensive numerical experiments were conducted in the course of our investigations. The results of these experiments are collected in Table~\ref{table:real} for the real case and Table~\ref{table:complex} for the complex case.  Unlike the case of  tight designs,  optimal weights for these configurations are generally not equal  and thus must be computed for each relevant  value of  $p$. Each table gives the minimal support size of a conjectured optimal point set: when a  configuration on the sphere is origin-symmetric, this minimal support size equals half of the size of the named configuration. For example, the icosahedron has twelve vertices, however $6$ vertices on one hemisphere suffice to give a minimizer of the $3$-frame energy on $\S^2$. We give additional details in Section~\ref{sec:ea1} for these conjectured minimizers of the $ p $-frame energies. Notably several of these configurations are not universally optimal, and further, several universally optimal configurations are nowhere to be found in this table. We discuss common features of minimizers in Section~\ref{sec:remarks}. More details on symmetry of measures  and relations between spheres and projective spaces may be found in Section~\ref{subsec:antip}. \\

Our experimental results together with Theorem \ref{t:mainsphere} lead us to believe that clustering of minimizers is a general phenomenon when $p$ is not an even integer.
\begin{conjecture}\label{conj:discrete}
 In all dimensions $d\ge 2$ and for all $p>0$ such that $p\not\in 2 \mathbb N$, the minimizing measures of the $p$-frame energy \eqref{eq:pframe} are discrete. 
 \end{conjecture}
 This conjecture is additionally  supported by the fact that discreteness of minimizers  is known for certain attractive-repulsive potentials on $\R^d$ \cite{carrillo2017geometry} and has been conjectured for some other potentials on the sphere, e.g.\ those appearing in \cite{finsterSupport2013}, see also Section~\ref{sec:causal}.  It is worth noting that   in the classical paper \cite{Bjorck1956}, it was shown that for $F(x,y) = - \| x-y\|^{\alpha}$ with $\alpha >2$ and any compact domain $\Omega \subset \mathbb R^d$, the energy minimizers  are discrete and their support consists of  at most $d+1$ points (just two antipodal points if $\Omega = \mathbb S^{d-1}$). Moreover,  in  \cite{carrillo2017geometry} discreteness has been established for mildly repulsive potentials, i.e. those that behave as $- \| x-y\|^{\alpha}$ with $\alpha >2$ when $\| x-y\|$ is small.  Observe that for the $p$-frame potential, we have  $|\langle x,y \rangle |^p \, \approx 1 - \frac{p}{2} \|x-y \|^2$ when $x$, $y\in \mathbb S^{d-1}$ are close, hence the $p$-frame energy falls into the endpoint case $\alpha=2$, and, according to the discussion above, this case is more subtle. 
 
While we  have yet to establish Conjecture~\ref{conj:discrete} and prove discreteness, in our companion paper \cite{bilyk2019energy} we show that on $\S^{d-1}$, whenever $p$ is not even, the support of the measure minimizing the $p$-frame potential necessarily has empty interior.

In addition to the conjectured discreteness of minimizers our initial study gave rise to surprisingly symmetric minimizers for $p$-frame energies, suggesting that further investigation might give new interesting spherical codes. While nearly all of the minimizing configurations arising from our numerical experiments  have appeared before in the coding theory literature,  we did however discover a new code in $\C^5$ of $85$ vectors which in turn gives a new bound for a minimal sized weighted projective $3$-design. We detail a construction of this code and  its properties in Section~\ref{subsec:new_quadrature}.

\begingroup
\fontsize{10pt}{12pt}\selectfont 
\begin{table}[h]
    \begin{center}
        \caption{Optimal and conjectured optimal configurations for $p$-frame energies on $\RP^{d-1}$. Energies are evaluated in most cases at the odd integer which is the midpoint of the interval given. The range $q-$ configurations are obtained as limiting configurations as $p$ tends to $q$ from below. For these configurations, the energy is evaluated for the even limit value. Among the configurations which are not tight, the $600$-cell is the only configuration which is proved to be optimal.}
        \begin{tabular}{ c  c  c  c  c  c }
\label{table:real}
            $d$ & $N$ & Energy &  Range of $p$ & Tight & Name \vspace{1 mm} \\ \hline 
    \rule{0pt}{3ex} $2$ & $N$ & $(*)$ &  $[2N-4,2N-2]$ & $t$ & regular $2N$-gon \vspace{0 mm} \\ \vspace{0 mm}
$d$ & $d$ & $1/d$ &  $[0,2]$ & $t$ & orthonormal basis \\ \vspace{0 mm}
            $3$ & $6$ & $0.241202265916660$ & $[2,4]$ & $t$
                & icosahedron \\ \vspace{0 mm}
            $3$ & $11$ & $0.142857142857143$ & $6-$ & ${ }$ & Reznick design \\ \vspace{0 mm}
            $3$ & $16$ & $0.124867143799450$ & $[6,8]$ & ${ }$ & icosahedron and dodecahedron \\ \vspace{0 mm}
            $4$ & $11$ & $0.125000000000000$ & $4-$ & ${ }$ & small weighted design \\ \vspace{0 mm}
            $4$ & $24$ & $0.096277507157493$ & $[4,6]$ & ${ }$ & $D_4$ root vectors \\ \vspace{0 mm}
            $4$ & $60$ & $0.047015486159502$ & $[8,10]$ & ${ }$ & $600$-cell \\ \vspace{0 mm}
            $5$ & $16$ & $0.118257675970387$ & $[2,4]$ & ${ }$ & hemicube \\ \vspace{0 mm}
            $5$ & $41$ & $0.061838820473855$ & $[4,6]$ & ${ }$ & Stroud design \\ \vspace{0 mm}
            $6$ & $22$ & $0.090559619406078$ & $[2,4]$ & ${ }$ & cross-polytope and hemicube \\ \vspace{0 mm}
            $6$ & $63$ & $0.042488105634495$ & $[4,6]$ & ${ }$ & $E_6$ and $E_6^{*}$ roots \\ \vspace{0 mm}
            $7$ & $28$ & $0.071428571428571$ & $[2,4]$ & $t$ & kissing $E_8$ \\ \vspace{0 mm}
            $7$ & $91$ & $0.030645893660944$ & $[4,6]$ & ${ }$ & $E_7$ and $E_7^{*}$ roots \\ \vspace{0 mm}
            $8$ & $36$ & $0.059098639455782$ & $3$ & ${ }$ & mid-edges of regular simplex \\ \vspace{0 mm}
            $8$ & $120$ & $0.022916666666667$ & $[4,6]$ & $t$ & $E_8$ roots \\ \vspace{0 mm}
            $23$ & $276$ & $0.011594202898551$ & $[2,4]$ & $t$ & equiangular lines \\ \vspace{0 mm}
            $23$ & $2300$ & $0.002028985507246$ & $[4,6]$ & $t$ & kissing Leech lattice \\ \vspace{0 mm}
            $24$ & $98280$ & $0.000103419439357$ & $[8,10]$ & $t$ & Leech lattice minimal vectors \\ \vspace{0 mm}
        \end{tabular}
    \end{center}
\end{table}
\endgroup
\begingroup
\fontsize{10pt}{12pt}\selectfont
\begin{table}[h]
    \begin{center}
        \caption{Optimal and conjectured optimal configurations for $p$-frame energies on $\CP^{d-1}$. The energies are evaluated at odd integers.}
        \begin{tabular}{ c c c c c c }
            \label{table:complex}
            $d$ &	$N$ &	Energy		&		Range of $p$	& Tight	& Name \vspace{1 mm} \\ \hline 
\rule{0pt}{3ex} $d$ & $d$ & $1/d$ &  $[0,2]$ & $t$ & orthonormal basis \vspace{0 mm} \\ \vspace{0 mm}
            $3$ &	$9$ &	$0.222222222222223$& $[2,4]$& $t$ & SIC-POVM	\\ \vspace{0 mm}
            % value is actually (-3+5*Sqrt[3])/(36*(8-6*Sqrt[2]+Sqrt[3]))
            $3$ &	$21$ &	$0.012610934678518$& $[4,6]$& ${ }$ & union equiangular lines	\\ \vspace{0 mm}
            $4$ &	$16$ &	$0.146352549156242$& $[2,4]$& $t$ & SIC-POVM	\\ \vspace{0 mm}
            $4$ &	$40$ &	$0.068301270189222$& $[4,6]$& $t$ & Eisenstein structure on $E_8$	\\ \vspace{0 mm}
            $5$ &	$25$ &	$0.105319726474218$& $[2,4]$& $t$ & SIC-POVM \\ \vspace{0 mm}
            $5$ &	$85$ &	$0.041997097378053$& $[4,6]$& ${ }$ & $O_{10}$ and $W(K_5)$ minimal vectors	\\ \vspace{0 mm}
            % value is actually 2(23+20*Sqrt[3])/2745
            $6$ &	$36$ &	$0.080272843473504$& $[2,4]$& $t$ & SIC-POVM	\\ \vspace{0 mm}
            $6$ &	$126$ &	$0.027777777777778$& $[4,6]$& $t$ & Eisenstein structure on $K_{12}$	\\ \vspace{0 mm}
            $d$ &	$d^2$ &	$\frac{1+(d^2-1)(1/(d+1))^{3/2}}{d^2}$ & $[2,4]$& $t$ 	& SIC-POVM (conjectured) \\ \vspace{0 mm}
        \end{tabular}
    \end{center}
\end{table}
\endgroup

Section \ref{sec:non} extends some of our results to non-compact settings.   In Section \ref{sec:mixed_volume} we apply the results of Theorem \ref{t:mainsphere} to the  problems of minimizing mixed volumes of convex bodies, and in Section \ref{sec:causal}  we apply the methods of linear programming, similar to those employed in Theorems \ref{t:mainsphere} and \ref{thm:tight}, to the optimization of energies related to {\it causal variational principles}, see \cite{finsterSupport2013}.

We would like to point out that in many papers, the term {\it $p$-frame potential} is usually used to denote the $p$-frame energy \eqref{eq:pframe} or its discrete counterpart. We find the term ``energy'' to be more appropriate in this context and reserve the term ``potential'' for the kernel $f(t)$  of the energy $I_f$.

\section{Geometry and functions on 2-point homogeneous spaces}\label{sec:prelim}
\subsection{Two-point homogeneous spaces} 
\label{subsec:2point}

For convenience, the above discussion mostly  assumed the underlying space to be the unit sphere $ \S^{d-1} $. This will no longer be the case, as our study concerns energy minimization on a broader class of spaces.
A metric space $ (\Omega, d)  $ is said to be \textit{two-point homogeneous}, if for every two pairs of points $ x_1, x_2 $ and $ y_1, y_2 $ such that $ d(x_1, x_2) = d(y_1, y_2) $ there exists an isometry of $ \Omega $, mapping $ x_i $ to $ y_i, $ $ i=1,2 $. It is known \cite{wangTwoPoint1952} that any such compact connected space is either a real sphere $ \S^{d-1} $, a real projective space $ \RP^{d-1} $, a complex projective space $ \CP^{d-1} $, a quaternionic projective space $ \mathbb{HP}^{d-1} $, or the Cayley projective plane $ \OP^2 $. Note that it suffices to consider $\FP^{d-1}$ for $ d > 2 $ only, as $ \FP^1 $ is just $\S^{\dim_{\R} \F} $ \cite[p. 170]{baez2002oct}, and so will not be separately considered in what follows. 

Below, $\Omega$ always refers to a compact connected $2$-point homogeneous space, equipped with the geodesic distance $ \theta $, normalized to take values in $ [0, \pi] $. We let $\sigma$ denote the unique probability measure invariant under the isometries of $\Omega$.

The first three types of projective spaces $ \{ \FP^{d-1} : \F=\R,\C, \mathbb H \} $ have a simple description: they may be represented as the spaces of lines passing through the origin in $\F^d$, 
\begin{equation}
    \label{eq:linemodel}
    x\F=\{x\lambda\ |\ \lambda\in \F\setminus\{0\}\}. 
\end{equation}
Observe that the isometry groups $ O(d) $, $ U(d) $, $ Sp(d) $ of the corresponding vector spaces $ \F^d $ act transitively on each space, and that the stabilizers of a line represented by $ x \in \F^{d} $ are $ O(d-1)\times O(1) $, $ U(d-1)\times U(1) $, and $ Sp(d-1)\times Sp(1) $, respectively. Thus one has \cite[p. 28]{wolf2007harmonic} the following quotient representations:
\begin{equation*} 
    \begin{aligned}
        \RP^{d-1}         &= O(d)/O(d-1)\times O(1),\\
        \CP^{d-1}         &= U(d)/U(d-1)\times U(1),\\
        \mathbb{HP}^{d-1} &= Sp(d)/Sp(d-1)\times Sp(1),
    \end{aligned}
\end{equation*} 
where we write $ O(d) $, $ U(d) $, $ Sp(d) $ for the groups of matrices $ X $ over the respective algebra, satisfying $ XX^* = I$.

Using the identification \eqref{eq:linemodel}, one can associate each element of   $ \FP^{d-1} $ ($ \F=\R,\C, \mathbb H $) with a unit vector $x\in \F^d$, $\| x\|= 1$, and we shall often abuse notation by doing so.   This gives, in addition to the Riemannian metric $ \theta $, another metric, the {\it chordal distance} between points $ x,y \in \Omega $, defined by
\[\rho(x,y)=\sqrt{1-|\langle x,y \rangle|^2},  \]
where $\displaystyle{\langle x, y \rangle=\sum\limits_{i=1}^{d} {x}_{i}\overline y_{i} }$ is the standard inner product in $\F^d$.   
The chordal distance $\rho(x,y)$  is related to the geodesic distance  $\vartheta(x,y)$ by the equation $$\cos \vartheta(x,y)=1-2\rho(x,y)^2=2|\langle x,y \rangle|^2-1.$$ 

Since the algebra of octonions is not associative, the line model of \eqref{eq:linemodel} fails, and instead a model given by Freudenthal \cite{freudenthal} is used to describe $\OP^{d-1}$. It is known \cite{baez2002oct} that only two octonionic spaces exist: $\OP^{1}$ and $\OP^{2}$, however $\OP^{1}$ is just $\S^{8}$, as noted above. 

$\OP^2$ can be described as the subset of $3\times 3$ Hermitian matrices $\Pi$ over $\O$, satisfying $\Pi^2=\Pi$ and $\text{Tr}\ \Pi=1$ \cite{skriganovPoint2017}. A metric for $\OP^2$ is then given by the Frobenius product, 
$$\rho(\Pi_1,\Pi_2)=\frac{1}{\sqrt{2}}\|\Pi_1-\Pi_2\|_{F}=\sqrt{1-\langle \Pi_1,\Pi_2 \rangle},$$
where $\langle \Pi_1,\Pi_2 \rangle=\text{Re}\ \text{Tr}\ \frac{1}{2}(\Pi_1\Pi_2+\Pi_2\Pi_1)$. This is the chordal distance on $\OP^2$ whereas the geodesic distance can be defined through $\sin{\frac{\vartheta(x,y)}{2}} = \rho(x,y)$, as in the above projective spaces. All $\Pi$ given as above may be written in the form 
$$ \begin{pmatrix} |a|^2 & a\overline{b} & a\overline{c} \\ b\overline{a} & |b|^2 & b\overline{c} \\ c\overline{a} & c\overline{b} & |c|^2 \end{pmatrix}, $$ 
where $|a|^2+|b|^2+|c|^2=1$ and $(ab)c=a(bc)$. This gives a representation of $\OP^2$ as the quotient $F_4/{\rm Spin}(9)$ \cite[p. 189]{baez2002oct}. 

One feature of spaces $ \Omega $ that allows for the application of linear programming methods is the existence of a decomposition of $L^2(\Omega, \sigma)$, the space of complex-valued square-integrable functions on $ \Omega $:
\[L^2(\Omega, \sigma)=\bigoplus\limits_{n\geq 0}V_n , \] 	
 where $ V_n $ are finite-dimensional irreducible representations of the isometry group of $ \Omega $ (see \cite{levenshteinDesigns1992}). Moreover, these are in correspondence with the eigenspaces  of the Laplace--Beltrami operator on $\Omega$  corresponding to the $n$-th eigenvalue in the increasing order. 
Let $Y_{n,k}$, $ k =1,\ldots,\dim V_n $, be an orthonormal basis in $V_n$. Because of the invariance of $ V_n $ and due to the two-point homogeneity of $\Omega$, the reproducing kernel for $ V_n $ only depends on the distance $\vartheta(x,y)$  between points \cite{venkovRes2001}. Furthermore, as a function of \[
    \tau(x,y) := \cos \vartheta(x,y).
\]  the reproducing kernel is a polynomial $C_n$ of degree $n$, which satisfies
\begin{equation}
    \label{eq:addition_formula}
    C_n(\tau(x,y))= \frac1{\dim V_n} \sum\limits_{k=1}^{\dim V_n} Y_{n,k}(x) \overline{Y_{n,k}(y)}. 
\end{equation}

\noindent Formula \eqref{eq:addition_formula} is known as the {\it addition formula}, and shows that functions $C_n$ are {\it positive definite} on $\Omega$, that is, 
\[\sum\limits_{1\leq i,j\leq k} c_i\overline{c}_j C_n(\tau{(x_{i},x_{j})})\geq 0 \] 
for all coefficients $c_1, ..., c_k \in \C$, and all $x_1, ..., x_k \in \Omega$.

The polynomials $C_n$ given by \eqref{eq:addition_formula} satisfy $ C_n(1) = 1 $ and are orthogonal with respect to the probability measure
\[
    d\nu^{(\alpha, \beta)}  =  \frac1{\gamma_{\alpha,\beta}} (1-t)^{\alpha} (1+t)^{\beta} dt,
\]
where $\alpha = (d-1) \dim_{\mathbb R} (\mathbb F)/{2}-1$ and
\begin{equation}
    \label{eq:jacobibeta}
    \beta = \begin{cases}
        \alpha, & \text{ if } \Omega = \S^{d-1};\\
        {\dim_{\R}(\F)}/{2} -1, & \text{ if } \Omega = \mathbb{FP}^{d-1},
    \end{cases}
 \end{equation} 
 and the normalization factor is given by
 \[ 
     \gamma_{\alpha,\beta} = 2^{\alpha +\beta +1}B(\alpha+1, \beta+1),
 \]
 where $B$ is the beta function. These polynomials, known as Jacobi polynomials (Gegenbauer polynomials in the special case when $\Omega = \mathbb S^{d-1}$), form an orthogonal basis in $L^2([-1,1],d\nu^{(\alpha, \beta)})$; equivalently, the span of $ C_n(\tau(x,y)) $, $ n\geq 0 $, is dense in the subset of  $ L^2(\Omega\times \Omega, \sigma\otimes\sigma) $ consisting  of functions that depend only on the distance between $ x $ and $ y $.

This allows for expanding functions from $L^2([-1,1],d\nu^{(\alpha, \beta)})$ in terms of $C_n$:
\[f(t)=\sum\limits_{n=0}^\infty \widehat{f}_nC_n(t), \quad \text{where}\quad \widehat{f}_{n}= {\dim V_n}\int_{-1}^{\ 1} f(t) C_n(t)\,d\nu^{(\alpha, \beta)}(t).\] 
As we have already done above, for a fixed space $ \Omega $ we will not indicate the dependence of polynomials $ C_n = C_n^{(\alpha,\beta)} $ on the indices $ \alpha $, $ \beta $. We refer to $ \widehat f_n $ as the Jacobi coefficients of the function  $ f $; the normalization $ C_n(1) = 1 $ used here is common in the coding theory community \cite{szego1975orth, levenshteinDesigns1992}.

\subsection{Energies on 2-point homogeneous spaces}
\label{sec:energies}
 
For the space of probability measures $\P(\Omega)$ supported on $\Omega$, and for a lower semi-continuous function $f:[-1,1]\rightarrow \R\cup \infty$, the $f$-energy integral is defined as the functional mapping $\mu$ to 
$$I_f(\mu)=\int_{\Omega}\int_{\Omega}f(\tau(x,y))d\mu(x)d\mu(y).$$
Observe that when $\Omega = \S^{d-1}$, we have $ \tau(x,y) = \cos \vartheta(x,y) = \langle x,y \rangle$ and the definition above coincides with \eqref{eq:muener}.

We start by introducing  the notion of positive definite functions, which plays an important role in energy minimization and for the linear programming bounds we derive later. Below $C[-1,1]=C_{\R}[-1,1]$  denotes the space of continuous real valued functions on the interval $[-1,1]$.

 \begin{definition}\label{def:pd}
 Let $f\in C[-1,1]$. We say that $f$ is {\it{positive definite}}  on $\Omega$ if for any $ x_1, ..., x_N \in  \Omega $  the matrix $\big[  f\big( \tau (x_i, x_j ) \big) \big]_{i,j=1}^N$ is positive semidefinite, i.e. for every collection $  c_1,\ldots,c_N  \in \mathbb C $ we have   \[
        \sum_{1\leq i,j \leq N} f( \tau(x_i, x_j) )c_i \overline{c_j} \geq 0.
    \] 
 \end{definition}
 We have already seen that the Jacobi polynomials $C_n$ are positive definite on $\Omega$, and so their positive linear combinations must also be. It is a classical fact that this implication can be reversed:
 \begin{proposition}\cite{bochner, schoenbergPositive1941, gangolli}\label{prop:schon}
    A function $ f\in C[-1,1] $ is positive definite on $\Omega$    if and only if $ \widehat{f}_n \geq 0 $  for all $ n \geq 0 $.
\end{proposition}

Next we show that positive definite functions $f$ give rise to $f$-energy integrals which are minimized over probability measures by the surface (or Haar) measure $\sigma$ on $\Omega$. This result appears in a number of papers, see for instance \cite{damelin2003energy, bilyk2018geodesic}. We adapt the proof given in \cite{bilyk2018geodesic} to our purposes. %, choosing to work with the real and imaginary parts of the functions $Y_{n,k}$ defined above. By a slight abuse of notation, we use the same notation for these functions.  

\begin{proposition}\label{prop:pdmin} Let $f\in C[-1,1]$, $f(t)=\sum\limits_{n=0}^\infty \widehat{f}_n C_n(t)$, and $\mu\in\P(\Omega)$. Then, the following are equivalent: \begin{enumerate}[(i)] \item $\widehat{f}_{n}\ge 0$ for all $n\geq 1$,
\item the surface measure $\sigma$ is a minimizer of $I_f$.
\end{enumerate}
Moreover,  $\sigma$ is the unique  minimizer of $I_f$ if and only if $\widehat{f}_n> 0$ for all $n\ge 1$. 
\end{proposition}
To prove this statement we use the following lemma, generalizing the  behavior of Fourier expansions with positive coefficients \cite{gangolli,Lyubich2009} to Jacobi expansions with the same property.
\begin{lemma} Assume that $f\in C[-1,1]$ has the Jacobi expansion $f(t)=\sum\limits_{n=0}^\infty \widehat{f}_n C_n(t)$ with  $\widehat{f}_n\geq 0$ for all $n\geq 1$. Then this expansion  converges uniformly and absolutely to $f$ on $[-1,1]$.
\end{lemma}
\begin{proof}[Proof of Proposition~\ref{prop:pdmin}] We first show that $\sigma$ is a minimizer of $I_{f}$. Assume that $\widehat{f}_n\geq 0$ for all $n\geq 1$. Then by the lemma above,  the Fubini theorem, and the addition formula,  we have 
\begin{align*} 
I_f(\mu)&=\sum\limits_{n=0}^\infty \widehat{f}_n \int_{\Omega}\int_{\Omega} C_n(\tau(x,y))d\mu(x)d\mu(y) \\
&=\sum\limits_{n=0}^\infty\frac{1}{\dim V_n}  \sum\limits_{k=1}^{\dim V_n} \widehat{f}_n\int_{\Omega}\int_{\Omega} Y_{n,k}(x) \overline{Y_{n,k}(y)} d\mu(x)d\mu(y) \\
&=\widehat{f}_0+\frac{1}{\dim V_n}\sum\limits_{n=1}^\infty b_{n,\mu}\ \widehat{f}_n  \ge \widehat{f}_0=I_{f}(\sigma).
\end{align*}
The last inequality holds since $b_{n,\mu}=\sum\limits_{k=1}^{\dim V_n} \left| \int_{\Omega}Y_{n,k}(x)d\mu(x) \right|^2\geq 0$. If $\widehat{f}_{n}> 0$ for all $n\geq 1$, then equality can be achieved above only if $\mu $ is orthogonal to all  spaces $V_n$, $n \geq 1$, which directly implies that $\mu = \sigma$. 

Let us assume that for some $m \in \mathbb{N}_0$, $\widehat{f}_n < 0$. For a fixed point $p \in \Omega$, we see that $Y_{n,1}(x) := C_n(\tau(x, p))$ is in $V_n$ and real-valued. Set $d\mu(x)=(1+\epsilon Y_{n,1}(x))d\sigma(x)$, where $\epsilon>0$ is sufficiently small so that $(1+\epsilon Y_{n,1}(x))\geq 0$ on $\Omega$. Orthogonality and the addition formula (or Funk-Hecke formula) give that for $Y\in V_{n},$ 
$$\int_{\Omega} f(\tau(x,y))Y(x)d\sigma(x) = \frac{1}{\dim(V_n)}\widehat{f}_n \overline{Y(y)} \text{ and } \int_{\Omega}Y (x)d\sigma=0.$$
Thus, 
\begin{align*} 
I_{f}(\mu)&=\int_{\Omega}\int_{\Omega}f(\tau(x,y))(1+\epsilon Y_{n,1}(x))(1+\epsilon Y_{n,1}(y))d\sigma(x)d\sigma(y) \\
&= I_{f}(\sigma)+ \frac{1}{\dim(V_n)}\int_{\Omega}\epsilon^2 \widehat{f}_n  Y^2_{n,1}(x) d\sigma(x) < I_{f}(\sigma),
\end{align*}
implying that $\sigma$ is not a minimizer for $I_{f}$. If $\widehat{f}_n = 0$ for some $n\geq 1$, the same argument shows that $I_f (\mu ) = I_f (\sigma)$, i.e. $\sigma $ is not the unique minimizer.  
\end{proof}\vskip3mm

The $p$-frame energies correspond to  taking $\Omega = \FP^{d-1}$ ($\F = \R$, $\C$, or $\mathbb H$) and  $f$  of the form 
\begin{equation}
    \label{eq:symmetrized_pkernel}
    f(t)=\bigg(\frac{1+t}{2}\bigg)^{\frac{p}{2}},
\end{equation}
because in this case, since $\tau(x,y) = \cos \vartheta(x,y)= 2|\langle x,y \rangle|^2-1$, we have  
$$ f (\tau(x,y)) = f \big( 2|\langle x,y \rangle|^2-1 \big) = |\langle x,y \rangle|^p.$$
 We shall  now prove that, whenever $p$ is an even integer,  these energies are minimized by the uniform measure on $\Omega$.   
 
 When $p=2k$ and  $\Omega = \FP^{d-1}$  ($\F=\R,\C,$ or $\mathbb H$), we have  that  $f(t)=2^{-k} \cdot (1+t)^{k}$ is a polynomial. It is standard to check that  this polynomial  is positive definite on $\Omega$: this could be done by checking that the coefficients in its Jacobi expansion are non-negative, but it would be  perhaps simpler to prove it as follows. Observe that,  since $C_0^{(\alpha,\beta)} (t) =1$ and $C_1^{(\alpha,\beta)} (t) = \frac{\alpha - \beta}{2(\alpha+1)}+ \frac{\alpha + \beta +2}{2(\alpha+1)}\cdot t$, we have that  
 $$ 1 +t =  \frac{2(\alpha+1)}{(\alpha + \beta +2) } \, C_1^{(\alpha,\beta)} (t) \,  + \, \frac{2(\beta +1)}{\alpha + \beta +2 }\, C_0^{(\alpha,\beta)} (t).$$
 Since  $\alpha +1 = \frac{d-1}2 \cdot \dim_{\mathbb R} (\mathbb F) >0$ and $\beta + 1 =  \frac{1}{2} \cdot \dim_{\mathbb R} (\mathbb F) >0$, we see that the function $1+t$ is positive definite on $\Omega$. The well known Schur's theorem on  Hadamard (elementwise) products of positive semidefinite matrices implies that if $g$ and $h$ are positive definite on $\Omega$, then so is their product $gh$, and, in particular, all integer powers $g^n$  are positive definite. Hence, the function $f(t) = 2^{-k}\cdot (1+t)^k$ is positive definite on $\Omega$, and therefore $I_f$ is minimized by the uniform surface measure $\sigma$. 

The minimal values of the $p=2k$ energy may be expressed in terms of elementary functions for each $\F$. These constants, $c_{\F}(d,k)$, are given below 

\begin{align*} c_{\F}(d,k) & = \frac{1\cdot 3\cdot 5 \dots (2k-1)}{d\cdot(d+2)\dots (d+2(k-1))},\ & \F=\R, \\
c_{\F}(d,k) & = 1/\binom{d+k-1}{k},\ & \F=\C, \\
c_{\F}(d,k) & = (k+1)/\binom{2d+k-1}{k},\ & \F=\mathbb{H}. \\
\end{align*}
  
When $p$ is not an even integer, the $p$-frame  energies are not positive definite, due to the appearance of negative terms in the Jacobi polynomial expansion of $f$, hence $\sigma$ does not  minimize the $p$-frame energy for $p \not\in 2 \N$, see Lemma 6.2.2 in \cite{Matzke}. % \textcolor{blue}{Add reference to Ryan's thesis, Lemma 6.2.2} 

\subsection{Designs}\label{s:des}
We now treat the topic of designs in the compact connected two-point homogeneous spaces $\Omega$. A finite, nonempty set (code) $\mathcal{C}\subset \Omega$ %, also known as a {\it code}, 
is called an {\it $M$-design} if 
\begin{equation}
    \label{eq:design}
    \frac{1}{|\mathcal{C}|}\sum\limits_{x\in\mathcal{C}} p(x)=\int_{\Omega} p(x)\, d\sigma(x)
\end{equation} 
holds for all polynomials $p$ of degree at most $ M $. %Here $d\sigma_{\Omega}(x)$ is the Haar (or surface) measure on $\Omega$. 
A relaxation of the above identity allows the configuration to be weighted, so that the equality
\begin{equation}
    \label{eq:wtdesign}
\sum\limits_{x \in \mathcal C} \omega_{x} p(x) =\int_{\Omega} p(x)\, d\sigma(x),
\end{equation} 
holds for some weights $\{\omega_{x}\}_{x\in \mathcal C}\subset \R_{\geq 0}$, satisfying $\sum_{x \in \mathcal C} \omega_{x}=1$, and all polynomials $p$ of degree at most $ M $. Such weighted formulas are called {\it cubature formulas} or {\it weighted designs}. In both of the above equations, it is understood that polynomials $p$ may be given explicitly as complex-valued functions which are polynomials in coordinates of $\F^d$, satisfying additionally  $p(\alpha x)=p(x)$, for $|\alpha|=1$, $ \alpha \in \mathbb F $, in the projective case.

The {\it strength} of a (weighted) design is the maximum value of $M$ for which identity~\eqref{eq:design} (accordingly,~\eqref{eq:wtdesign}) holds.  
An $ M $-design can be equivalently defined as a configuration $ \mathcal C\subset \Omega $, for which
\begin{equation*} 
    \sum\limits_{x,y\in\mathcal{C}}C_n(\tau(x,y))=0\quad \text{ for }1\leq n \leq M.
\end{equation*} 
Equivalently, $ \mathcal C $ is an $M$-design in $\Omega$ if and only if it satisfies  
$$\sum\limits_{x\in\mathcal{C}} Y(x)=0\ \text{ for}\ Y\in\bigoplus\limits_{n=1}^{M} V_n.$$
Similar definitions can be given for weighted designs.

Linear programming bounds \cite{delsarteSpherical1977} imply exact constraints on the size of {\it tight designs}, configurations which, in addition to being $M$-designs, have the smallest possible number of pairwise distances between their elements, for a design of strength $M$. The exact definition may be given as follows.
\begin{definition}\label{d:tight}
A discrete set $ \mathcal C \subset \Omega $ is called a \textit{tight $M$-design} if one of the following conditions is satisfied. 
\begin{enumerate}
    \item[(i)] $ \mathcal C $ is a design of strength $ M=2m-1 $ and there are $ m $ distances between its distinct elements, including at least one pair diameter apart;
    \item[(ii)] $ \mathcal C $ is a design of strength $ M=2m $ and there are $ m $ distances between its distinct elements.
\end{enumerate} 
\end{definition}

Table \ref{tab:tsphdes} provides a list of known tight spherical designs (see, e.g.,  \cite{cohnUni2007}), as well as the $600$-cell, which is not a tight design, but will be of interest in Section \ref{sec:600cell}.  Each arrangement labeled `kissing' is the kissing configuration of a set. By centering non-overlapping congruent spherical caps of maximal radius at each point in a given code, the resulting points of tangency on a given cap form a spherical code in a lower dimensional space which we call the kissing configuration for that set.

Tight spherical designs with $d\geq 3$ and $M\geq 4$ may only exist for $M=4,5,$ and $7$ with the one exception of the spherical $11$-design formed by the Leech lattice minimal vectors \cite{bannaidam79, bannaidam80}. The problem of finding tight spherical $5$-designs is the same as that of finding maximal equiangular tight frames, and it is known that existence of a  tight spherical $5$-design in $\mathbb{S}^{d-1}$ is possible only for $d=1,2,3$ and for dimensions of the form  $d=(2k+1)^2-2$, where $k\geq 1$; see \cite{bannaidam79, bannaidam80, delsarteSpherical1977, LemSei73} for details on how these conditions arise. A direct correspondence with such spherical designs and regular graphs has long been recognized \cite{seidel2graph}, and, in connection, it is known that for infinitely many values of $k$, a tight spherical $5$-design cannot exist in dimension $d = (2k+1)^2-2$ \cite{bannaimumvenk, makhnev}.

\begin{table}[h] 
    \begin{center}
        \caption{A list of known tight spherical designs (with the $600$-cell). Here $M$ denotes the strength of the design, $d$  the dimension of the ambient space $\mathbb{R}^{d}$, and $N$ is the size of the design.}
        \label{tab:tsphdes}
        \begin{tabular}{ c  c  c  c  c } 
            $d$  & $N$ & $M$ & Inner products &  Name \vspace{1 mm} \\ \hline

\rule{0pt}{3ex}
$d$ & $2$ & $1$ & $\pm 1$ & Pair of antipodal points\\

$d$ & $d+1$ & $2$  & $-1/d, 1 $ & Regular simplex \\

$d$ & $2d$ & $3$  & $0,\pm 1 $ & Cross-polytope \\ 

$2$ & $N$ & $ N -1$ & $ \cos\left( 2j \pi/N \right)$, $0 \leq j \leq N/2$ & Regular $N$-gon \\ 

$3$ & $12$ & $5$ & $\pm 1/\sqrt{5}, \pm 1$ & Icosahedron \\ 

$4$ & $120$ & $11$ & $0, (\pm 1 \pm \sqrt{5})/4, \pm 1/2, \pm 1$ & $600$-cell \\ 

$6$ & $27$ & $4$ & $-1/2, 1/4, 1$ & Schl\"{a}fli config \\

$7$ & $56$ & $5$ & $ \pm 1/3, \pm 1$ & Kissing config for $E_{8}$ \\ 

$8$ & $240$ & $7$ & $0, \pm 1/2, \pm 1$ & $E_8$ root system \\ 

$22$ & $275$ & $4$ & $-1/4, 1/6, 1$ & McLaughlin config \\

$23$ & $552$ & $5$ & $\pm 1/5, \pm 1$ & Equiangular lines \\ 

$23$ & $4600$ & $7$ & $0, \pm 1/3, \pm 1$ & Kissing config for $\Lambda_{24}$\\ 

$24$ & $196560$ & $11$ &  $0, \pm \frac{1}{4}, \pm \frac{1}{2}, \pm 1$ & Leech Lattice $\Lambda_{24}$ minimal vectors
   
        \end{tabular}
    \end{center}
\end{table} 

Table \ref{tab:tprojdes} lists all known tight projective designs (see \cite{CKM}), except those for the spaces $\mathbb{FP}^1$ , as those are congruent to real spheres. Identifying tight projective designs is simple in the real setting. Tight  spherical designs of odd strength must be centrally symmetric \cite{delsarteSpherical1977}, and by choosing points from each antipode in an odd tight design, one arrives at a real projective tight design. Thus, all tight designs of odd strength in Table \ref{tab:tsphdes} correspond to entries in Table \ref{tab:tprojdes}.

For the other projective spaces, the vertices of a cross-polytope (i.e. an orthonormal basis in the projective space) always provide a tight $1$-design, as they did in $\mathbb{RP}^{d-1}$.  However, unlike the real case, it is known that no tight $t$-designs exist in the complex or quaternionic setting whenever $M\geq 4$ and $d \geq 3$ \cite{bannaihogg89, hogg89, Lyubich2009}. In the complex setting, tight $2$-designs, also known as \textit{symmetric, informationally complete, positive operator-valued measures (SIC-POVMs)}, are known to exist for $d \leq  16$, $d = 19, 24, 28, 35, 48$, and numerical experiments suggest that they may exist in every dimension \cite{AppBenBriEriGraLar14,  RBSC, SG, zau99}.  With exception of the $(3,15)$ quaternionic and $(3,27)$ octonionic designs from \cite{CKM}, explicit constructions are readily found for the other designs mentioned in Table \ref{tab:tprojdes} \cite{hoggar1982t}.

\begin{table}[h] 
    \begin{center}
        \caption{A list of parameters for which projective tight designs are known to exist (besides  designs in $\mathbb{FP}^1$ for $\mathbb{F} \neq \mathbb{R}$). Here $M$ denotes the strength of the design, $d$  the dimension of the ambient space $\mathbb{F}^{d}$, and $N$ is the size of the design. For SIC-POVMs, these configurations exist for certain values of $d$, and may or may not exist for all values.}\label{tab:tprojdes}
        \begin{tabular}{ c  c  c  c  c  c } 
            $d$  & $N$ & $M$ & $| \langle x, y \rangle|^2$ & $\mathbb{F}$ & Name \vspace{1 mm} \\ \hline \rule{0pt}{3ex}

$d$   & $d+1$  & $1$ & $0, 1$ & $\mathbb{R}$ & Cross-polytope/ONB  \\

$2$   & $N$  & $N-1$ & $\cos^{2}( \pi j/ N)$, $1 \leq j \leq N$ & $\mathbb{R}$ & Regular $2N$-gon   \\ 

$3$   & $6$  & $2$ & $1/5, 1$ & $\mathbb{R}$ & Icosahedron   \\

            $7$  & $28$  & $2$ & $1/9, 1$ & $\mathbb{R}$  & Kissing configuration for $E_8$\\
            
            $8$   & $120$  & $3$ & $0, 1/4, 1$ & $\mathbb{R}$  & Roots of   $E_8$ lattice\\
             
            $23$   & $276$  & $2$ & $1/25, 1$ &  $\mathbb{R}$  & Equiangular lines \\
            
            $23$   & $2300$  & $3$ & $0, 1/9, 1$ &  $\mathbb{R}$ & Kissing configuration for $\Lambda_{24}$ \\
            
            $24$   & $98280$  & $5$ & $0, {1}/{16}, {1}/{4}, 1$ & $\mathbb{R}$ & Minimal vectors of $\Lambda_{24}$ \vspace{1 mm} \\ \hline

\rule{0pt}{3ex}
$d$   & $d+1$  & $1$ & $0, 1$ & $\mathbb{C}$ & Cross-polytope/ONB  \\

$d$ & $d^2$ & $2$ & $1/(d+1), 1$ & $\mathbb{C}$ & SIC-POVM \\

$4$   & $40$  & $3$ & $0,{1}/{3}, 1$ &  $\mathbb{C}$ & Eisenstein structure on $E_{8}$  \\ 
            $6$   & $126$  & $3$ & $0,1/4, 1$ &  $\mathbb{C}$ & Eisenstein structure on $K_{12}$ \vspace{1 mm} \\ \hline

\rule{0pt}{3ex}
$d$   & $d+1$  & $1$ & $0, 1$ & $\mathbb{H}$ & Cross-polytope/ONB  \\

$3$   & $15$  & $2$ & $2/7, 1$ &  $\mathbb{H}$ & Equiangular lines  \vspace{1 mm} \\

$5$   & $165$  & $3$ & $0,1/4, 1$ &  $\mathbb{H}$ & Quaternionic reflection group  \vspace{1 mm} \\ \hline

\rule{0pt}{3ex}
$3$   & $d+1$  & $1$ & $0, 1$ & $\mathbb{O}$ & Cross-polytope/ONB  \\

$3$   & $27$  & $2$ & $2/13, 1$ & $\mathbb{O}$ & Equiangular lines \vspace{1 mm} \\

$3$   & $819$  & $5$ & $0,1/4,1/2, 1$ & $\mathbb{O}$ & Generalized hexagon\\
& & & & & of order $(2,8)$   
        \end{tabular}
    \end{center}
\end{table} 

A weaker property of a design is sharpness, which will not play a role here. The paper \cite{cohnUni2007} proves that sharp designs, and tight designs in particular, are minimizers for discrete minimization problems with absolutely monotone kernels. A similar approach allows us to show that tight designs are optimal for the continuous $p$-frame energy.

\subsection{Antipodal symmetry}
\label{subsec:antip}

We observe that the energy $I_f $ on the sphere $\S^{d-1}_\F$, $\mathbb{F} = \mathbb{R}, \mathbb{C}, \mathbb{H}$, for the kernels $f$ with $f(\langle x, y \rangle)=f(|\langle x, y \rangle|)$ remains the same after averaging over unit multiples of vectors in the support of $\mu$. Let $U(\F)$ be the set of units in $\F$, $U(\F)=\{c\in \F\ :\ |c|=1\}$, and $\eta$ be the uniform measure on $U(\F)$. If one defines, for positive Borel measure $\mu$ on the sphere $\mathbb{S}_{\mathbb{F}}^{d-1}$ and Borel sets $B \subset \S_\F^{d-1}$,
$$\nu(B)=\frac{1}{\eta(U(\F))}\int_{U(\F)} \mu(cB)d\eta(c),$$
then $I_{f}(\nu)=I_{f}(\mu)$ for potential functions $f$ as above. This is the primary reason it is natural to consider projective spaces $\F\mathbb{P}^{d-1}$ as the optimization spaces for $p$-frame energies, as opposed to the spheres,  in the cases when the elements $x\in \F\mathbb{P}^{d-1}$ may be represented by unit vectors in $\F^d$. 

This  discussion shows that a minimizing measure on the sphere for $I_{f}$, with $f$ as above, can be taken to be symmetric, and that the problem of minimizing over symmetric measures on spheres is equivalent to minimizing energy over projective spaces. In particular, this explains part \eqref{t1} of Theorem~\ref{t:mainsphere}, since tight spherical $(2M+1)$-designs are necessarily symmetric \cite{delsarteSpherical1977}  and hence correspond  to tight real projective $M$-designs. 

\section{Optimality of tight designs for  kernels absolutely monotonic to degree \texorpdfstring{$M$}{M}} 
\label{sec:tight}

\subsection{Linear programming}
\label{subsec:linprog2}

The main goal of this section is to show that for those dimensions and values of $t$ for which tight designs exist, they are the global minimizers of the $ p $-frame energies for intervals of $ p $ between consecutive even integers. We will use linear programming bounds to  this end. 

The linear programming method provides bounds for optima in various optimization problems, and its use   is often aided by computational tools, where a problem is approximated by a finite-dimensional or discretized counterpart, then solved with a computer. It is surprising that this simple method provides optimal bounds often. 
This technique applies to all the $2$-point homogeneous spaces $\Omega$ described above.

Our application of the method can be summed up in the following lemma, which is a measure-theoretic counterpart of the linear programming bound of Delsarte and Yudin \cite{delsarte1973algebraic,yudinMin1992}.

\begin{lemma}\label{lem:lp}
Let $h \in C[-1,1] $ be  a positive-definite function, i.e. $h(t)=\sum\limits_{n=0}^\infty \widehat{h}_n C_n(t)$ and $\widehat{h}_n\geq 0$ for all $n\geq 0$. 
\begin{enumerate}[(i)]
\item \label{lp1}Assume that  $h(t)\leq f(t)$ for all $t\in [-1,1]$, then for any $\mu \in \mathcal{P}(\Omega)$, $$I_{f}(\mu)\geq \widehat{h}_0 = I_h (\sigma).$$
\item\label{lp2} Assume further that $h$ is a polynomial of degree $k$ and that there exists a $k$-design $\mathcal C \subset \Omega$ such that  $h(t) = f (t)$ for each $t\in \{ \tau (x,y): \, x,y \in \mathcal C \}$. Then for any $\mu \in \mathcal{P}(\Omega)$, $$I_{f}(\mu)\geq  I_f \Big(\frac1{|\mathcal C|} \sum_{x\in \mathcal C} \delta_x \Big),$$
i.e. $I_f$ is minimized by the uniform distribution on $\mathcal C$. 
\end{enumerate}
\end{lemma}
\begin{proof} For the first part observe that
$$ I_f(\mu) \geq I_h( \mu) \geq I_h(\sigma) = \widehat{h}_0,$$ where the first inequality follows from the fact that  $f\geq h$, while the second one is due to  Proposition \ref{prop:pdmin}, since $h$ is positive definite. 

For the second part, one can continue  as follows
$$ I_h (\sigma ) =  I_h \Big(\frac1{|\mathcal C|} \sum_{x\in \mathcal C} \delta_x \Big) =  I_f \Big(\frac1{|\mathcal C|} \sum_{x\in \mathcal C} \delta_x \Big).$$ The first equality follows from the fact that $\mathcal C$ is a $k$-design, and the second one from the fact that $f$ and $h$ coincide on the set $\{ \tau (x,y): \, x,y \in \mathcal C \}$. Together with part \eqref{lp1}  this proves the statement in part \eqref{lp2}. 
\end{proof}

This lemma provides insights in two different ways for how the linear programming method can be applied.

If a candidate $\mathcal C$ is available, one can apply part \eqref{lp2} of Lemma \ref{lem:lp} by  constructing a polynomial $h\le f$ as a Hermite interpolant of the function $f$  at the points of $\{ \tau (x,y): \, x,y \in \mathcal C \}$. This reasoning, which lies behind the proof of Theorems \ref{thm:tight} and \ref{t:mainsphere},  explains the appearance of  tight designs: indeed, the number of elements in the set of interpolation points (i.e. distinct distances between the points of $\mathcal C$) determines the degree of the interpolant $h$ -- hence one wants a design of high strength, but with few mutual distances. 

The same reasoning as above applies to the emergence of sharp designs as universally optimal sets in \cite{cohnUni2007}, and it also explains why this slightly weaker notion does not suffice for our purposes: since we are working with general measures rather than point sets with fixed cardinality, we cannot avoid interpolating at the point $t=1$, which requires a design of higher strength.  The main technical difficulty in this setting is proving positive definiteness of the Hermite interpolating polynomial $h$. We take this approach  to Theorem \ref{thm:tight} and carry out the technicalities in  Sections \ref{subsec:3.2}--\ref{subsec:tight}.

If a suitable candidate is not available, one can still rely on part \eqref{lp1} of Lemma \ref{lem:lp}  and attempt to   optimize the value of the energy $I_h (\sigma)$  over auxiliary positive definite polynomials $h$, obtaining a lower bound for the energy over all probability measures.  If the degree of an auxiliary function $h$ is bounded by $D$, we have $D+1$ non-negative variables $\widehat{h}_i$, $0\leq i\leq D$, and infinitely many linear constraints $h(t)\leq f(t)$ for all $t\in [-1,1]$. In order to get the best possible lower bound, we need to maximize $\widehat{h}_0$ given these linear conditions.

 This problem is, generally, intractable as a linear optimization problem. However, when $f$ is a polynomial, the condition $f(t)-h(t)\geq 0$ for all $t\in [-1,1]$ may be represented as a finite-size positive semi-definite constraint on the coefficients $\widehat{h}_i$. In particular, the polynomial inequality may be rewritten as a sum-of-squares optimization problem (see, for instance, \cite{nesterov2000squared}) and thus solved as a semi-definite program.

By using sum-of-squares optimization described above, we obtain  lower bounds on  the $p$-frame energies over measures on projective spaces when $p$ is an odd integer. A table of such bounds for real projective spaces $\RP^{d-1}$, $3\leq d \leq 24$, and $p=3, 5, 7$, is shown in Table~\ref{table:numlp} in the Appendix. The concrete bounds are computed by a series of steps. For the first step, we  fix the degree $D$ of the auxiliary polynomial and solve the sum-of-squares problem. The numerical solver outputs a polynomial which is feasible up to a small tolerance. By rounding coefficients, it is then possible to obtain polynomials which are less than $f$ and positive definite.

Since the choice of the maximal degree $D$ is arbitrary, not much is lost by rounding, and our bounds in the appendix are thus rounded down to four significant figures. The last condition $f-h\geq 0$ can be checked using interval arithmetic, or by hand. We include the coefficients of the auxiliary polynomials  in the supplementary files of the arXiv version of this paper. The polynomials used for $p=3$ and $p=5$ are of degree $D=6$, while the polynomials for $p=7$ are of degree $D=8$.

It is interesting to compare the values of conjectured energy minimizers with the lower bounds obtained using the approach above. We make comparison of these bounds in Table~\ref{table:compare} below for all conjectured optimizers from Tables \ref{table:real} and \ref{table:complex}: observe that the values are indeed close, which motivates our conjectures about the minimizers.  Tight designs are excluded from this table since for them the lower and the upper bounds coincide as we will show below in Theorem \ref{thm:tight}. 
 
\begin{table}[h]
    \begin{center}
        \caption{Comparison of $p$-frame energies for conjectured optimal configurations on $\RP^{d-1}$ and $\CP^{d-1}$ with LP lower bounds. Energies are evaluated at the odd integer midpoint of the conjectured optimality interval.}
        \begin{tabular}{ c  c  c  c  c   c }\label{table:compare}
            $d$ & $\F$ & Energy & LP bound & $p$ & Name \vspace{1 mm} \\ \hline 
            \rule{0pt}{3ex} $3$ & $\R$ & $0.1249$ & $0.1248$ & $7$ & icosahedron and dodecahedron \vspace{1 mm} \\ \vspace{1 mm}
            $4$ & $\R$ & $0.09628$ & $0.09607$ & $5$ & $D_4$ root vectors \\ \vspace{1 mm}
            $5$ & $\R$ & $0.1183$ & $0.1170$ & $3$ & hemicube \\ \vspace{1 mm}
            $5$ & $\R$ & $0.06184$ & $0.06169$ & $5$ & Stroud design \\ \vspace{1 mm}
            $6$ & $\R$ & $0.09056$ & $0.08970$ & $3$ & cross-polytope and hemicube \\ \vspace{1 mm}
            $6$ & $\R$ & $0.04249$ & $0.04240$ & $5$ & $E_6$ and $E_6^{*}$ roots \\ \vspace{1 mm}
            $7$ & $\R$ & $0.03065$ & $0.03060$ & $5$ & $E_7$ and $E_7^{*}$ roots \\ \vspace{1 mm}
            $8$ & $\R$ & $0.05910$ & $0.05852$ & $3$ & mid-edges of regular simplex \vspace{1 mm} \\ \hline
            \rule{0pt}{3ex} $3$ & $\C$ &	$0.01261$& 0.01258 & $5$ & union equiangular lines	\\ \vspace{1 mm}
            $5$ & $\C$ &	$0.04200$& 0.04184 & $5$ & $O_{10}$ and $W(K_5)$ minimal vectors	\\ 
        \end{tabular}
    \end{center}
\end{table}

\subsection{Properties of orthogonal polynomials}\label{subsec:3.2}
Recall that, for fixed $ \Omega $, we write simply
$C_n(t) = C_{n}^{(\alpha, \beta)}(t)$ with $ C_n(1) =1 $. In some of the arguments in Section~\ref{subsec:tight} we will instead use the monic polynomials proportional to $ C_n $; we therefore introduce notation $ Q_n(t) = Q_n^{(\alpha, \beta)}(t) $ for these Jacobi polynomials.  

In this subsection we collect several results about orthogonal polynomials relevant to the proof of our main theorem. 
Fix a space $ \Omega $, and let $ \alpha $ and $ \beta $ be the corresponding parameters of the associated Jacobi polynomials. According to Proposition \ref{prop:pdmin}, a function being positive definite on $\Omega$  is equivalent to having positive coefficients in the Jacobi expansion in terms $ Q_n^{(\alpha, \beta)} $. 

It will be  useful to consider \textit{adjacent} Jacobi polynomials, defined as one of the three sequences $ Q^{k,l}_n = Q_n^{(\alpha+k, \beta+l)} $ with $ k,l \in \{ 0,1 \} $, $ k+l >0 $. Specifically, we will need the following corollary which comes out of representing $ Q^{1,0}_n $ through $ Q^{0,0}_n $ \cite[equation (3.4)]{levenshteinDesigns1992}:
\begin{proposition}
    \label{prop:left_adjacent}
    Adjacent Jacobi polynomials $ Q_n^{1, 0} $ are positive definite on $\Omega$.
\end{proposition} 
On the other hand, adjacent polynomials $ Q^{1,1}_n $, defined as orthogonal with respect to the measure $ (1-t^2)\,d\nu^{(\alpha, \beta)} $, are not positive definite. The following property, a special case of the strengthened Krein condition \cite[Lemma 3.22]{levenshteinUniversal1998}, can serve as a substitute.
\begin{lemma}
    \label{lem:Krein}
$(t+1) Q^{1,1}_n(t) $ are positive definite on $\Omega$  for $ n \geq 0 $.
\end{lemma}

\begin{proof}
For all $n \in \N_0$, $(t+1)Q^{1,1}_n$ is orthogonal to all polynomials of degree less than $n$ with respect to the measure $(1-t) d\nu^{(\alpha, \beta)} =c_{\alpha,\beta} d\nu^{(\alpha+1, \beta)} $, so it can be expressed through the orthogonal polynomials corresponding to $ d\nu^{(\alpha+1, \beta)} $ as
$$(t+1)Q^{1,1}_n(t) = Q^{1, 0}_{n+1}(t) + b Q^{1, 0}_n(t),$$
for some constant $b$. Since all the roots of $ Q^{1, 0}_{n} $ lie in $ (-1,1) $,  $\text{sgn\,} Q^{1, 0}_{n}(-1) = (-1)^n$. Substituting $ t = -1 $ in the last equation gives $ Q^{1, 0}_{n+1}(-1) + b Q^{1, 0}_n(-1) = 0$, and so $b \geq 0$. By Proposition~\ref{prop:left_adjacent}, each $ Q^{1, 0}_n(t) $ is positive definite, and thus $(t+1)Q^{1,1}_n(t)$ is also positive definite. \end{proof}

Lastly, we will need the strict positive-definiteness of polynomials annihilated by subsets of roots of $ p_n + \gamma p_{n-1} $. We recall the following result.
\begin{proposition}[{\cite[Theorem 3.1]{cohnUni2007}}]
    \label{prop:subproducts}
    Consider  a sequence of orthogonal monic polynomials $ p_0(t),$  $ p_1(t), $  $ p_2(t), \ldots $, such that $\deg{p_k} = k$ for all $k \in \mathbb{N}_0$, and let $ t_1 <\ldots < t_n $ be the zeros of $ p_n + \gamma p_{n-1} $ for some fixed $ \gamma $. Then the polynomials
\[ 
    \prod_{i=1}^k (t-t_i), \qquad 1 \leq k < n,
\]
can be represented as a linear combination of $  p_0(t), p_1(t),\ldots, p_n(t) $ with positive coefficients.  
\end{proposition} 

\subsection{Hermite interpolation}

Let $f \in C^{K}[a,b]$, for some $K \in \N_0 $, and let a collection $t_1 < \ldots < t_m \subset [a,b]$, as well as positive integers $k_1, \ldots, k_m$ be given with $$ \max \{ k_1,\ldots, k_m \} \leq K+1.$$
There exists a polynomial $ p $ of degree less than $D = \sum_{i=1}^{m} k_i$, such that for $1 \leq i \leq m$ and $0 \leq k < k_i$,
$$p^{(k)}(t_i) = f^{(k)}(t_i).$$ 
Such a $ p $ is called the \textit{Hermite interpolating polynomial} of $ f $; it always exists and is unique because the linear map that takes a polynomial $p$ of degree less than $D$ to 
$$ (p(t_1), p'(t_1), \ldots, p^{(k_1 - 1)}(t_1), p(t_2), p'(t_2), \ldots, p^{k_m -1}(t_m))$$
is bijective.

It is convenient to organize both the collection $t_1 < \ldots < t_m$ and the orders of derivatives $k_1, \ldots, k_m$ into a polynomial $ g (t)  $.
Given such a polynomial 
$$g(t) = \prod_{i=1}^m (t - t_i)^{k_i} , $$
where $ D = \deg(g) \geq 1$, we write $H\left[f,g\right]$ for the interpolating polynomial of degree less than $ D $ that agrees with $f$ at each $ t_i $ to the order $ k_i $. Similarly, we let
$$ Q[f,g](t) = \frac{f(t) - H\left[f,g\right](t)}{g(t)},$$
be the \textit{divided difference} associated with the polynomial $ g $.  
Under the above hypotheses, for every $t \in [a,b]$ and a collection $ t_1 < t_2 < \ldots < t_m  $ as above, there exists $\xi \in (a,b)$ such that $ \min(t, t_1) < \xi < \max(t, t_m)$, and 
\begin{equation}
    \label{eq:div_difference}
    Q[f,g](t)  =  \frac{f^{(D)}( \xi)}{D!}.
\end{equation}

Enumerate the roots of $ g $ with multiplicities in increasing order, and denote these by $ s_{j}, 1\leq j \leq D  $, where $ s_j \leq s_{j+1} $. Let $ g_n $ be the polynomial annihilated on the first $ n $  elements of the  sequence $ s_1,\ldots, s_D $:
\[
    g_n(t) = \prod_{j=1}^n (t - s_j), \qquad 1 \leq  n \leq D.
\]
The usual assignment of the empty product applies here: $ g_0(t) = 1 $.

By the Newton's formula \cite[Chapter 4.6--7]{devoreConstructive1993}, the Hermite interpolating polynomial $ H\left[f,g\right] $  can be represented as 
\begin{equation}
    \label{eq:newton_formula}
    H\left[f,g\right](t) = f(s_1) + \sum_{j=1}^{D-1} g_{j}(t)\, Q[f, g_j](s_{j+1}).
\end{equation} 
The relevant property of the $ p $-frame kernel $ \left(\frac{s+1}2 \right)^{p/2} $ considered on a projective space $ \FP^{d-1} $ (for $\mathbb{F} = \mathbb{R}, \mathbb{C}, \mathbb{H}$), is that its first several derivatives are nonnegative on $ (-1, 1) $, followed by a negative one. Positivity of the derivatives implies, due to \eqref{eq:div_difference}, that the divided differences in  formula \eqref{eq:newton_formula} for the $ p $-frame kernel are nonnegative. It will be convenient to introduce notation for this number of nonnegative derivatives of a function.  

\begin{definition}
Let $f \in C^{M}(a,b) $. We say that $f$ is \textbf{absolutely monotonic of degree $M$} if $f^{(k)}(t) \geq 0$ for $0 \leq k \leq M$ and $t \in (a,b)$. If these derivatives are positive, we say that $F$ is \textbf{strictly absolutely monotonic of degree $M$}.
\end{definition}

 The usefulness of this new class of functions lies in that the Hermite interpolant of an absolutely monotonic function $ f $ of degree $ M $ with $ (M+1) $st derivative negative, will stay below $ f$, as shown in the following observation \cite{yudinMin1992}.

\begin{lemma}\label{lem:remaind}
    Let $ f:[-1,1] \to \mathbb R $ be absolutely monotonic of degree $ M $, and $ f^{(M+1)}(t) \leq 0 $ for all $ t \in (-1, 1) $. If the roots of a polynomial $ g $ of degree $ M+1 $ are contained in $ [-1, 1] $, and in addition $ g(t) \leq 0 $ for $ t \in [-1, 1] $, then, 
    \[
        f(t) \geq H[f,g](t), \qquad t \in [-1, 1].
    \]
\end{lemma}
\begin{proof}
    \label{lem:hermite}
    According to \eqref{eq:div_difference}, there exists $\xi \in (-1,1)$ such that $ \min(t, t_0) < \xi < \max(t, t_M)$, where the roots of $ g $ are $ t_0 \leq \ldots \leq t_M $, and 
    \[
        f(t) - H[f,g](t) = \frac{f^{(M+1)}( \xi)}{(M+1)!} g(t) .
    \]
    The expression on the right is nonnegative, so the conclusion of the lemma follows. \end{proof}

\subsection{Optimality of tight designs}
\label{subsec:tight}

As above, $\Omega$ is a compact, connected two-point homogeneous space and $Q_0, Q_1, Q_2, \ldots$ are the corresponding orthogonal polynomials. Recall that $ Q_n $ are orthogonal with respect to the measure $d\nu^{(\alpha, \beta)}  =  \frac1{\gamma_{\alpha,\beta}} (1-t)^{\alpha} (1+t)^{\beta} dt$, where the parameters $ \alpha $, $ \beta $ are chosen as in Section~\ref{subsec:2point}. The main result of this section is the following.
\begin{theorem}
    \label{thm:tight}
Let $f$ be absolutely monotonic of degree $M$, with $f^{(M+1)}(t) \leq 0$ for $t \in (-1,1)$. Then for a tight $ M $-design $ \mathcal C $,
$$ \mu_{\mathcal C} = \frac{1}{|\mathcal C|} \sum_{x \in \mathcal C} \delta_{x} $$
 is a minimizer of
$$ I_{f}(\mu) = \int_{\Omega} \int_{\Omega} f( \tau(x,y) ) \,d \mu(x) d \mu(y)$$
over $\mathcal P(\Omega)$, the set of probability measures on $\Omega$.
\end{theorem}
First of all, we show that this statement implies part \eqref{t2} of Theorem \ref{t:mainsphere}. 

\begin{proof}[Proof of part \eqref{t2} of Theorem \ref{t:mainsphere}]
Recall that, according to \eqref{eq:symmetrized_pkernel}, the $p$-frame energy on $\mathbb S_{\mathbb F}^{d-1}$ corresponds to the kernel   $f(t)=(\frac{1+t}{2})^{\frac{p}{2}}$  in the projective setting $\Omega = \FP^{d-1}$.  One can easily check that  $f^{(\lceil p/2 \rceil+1)}(t)\le 0$, $-1<t<1$, and all derivatives of smaller order are nonnegative. Thus  Theorem \ref{thm:tight} applies with $M = \lceil p/2 \rceil$, i.e.  tight projective $M$-designs minimize $I_f$ on $\FP^{d-1}$  for $2M - 2 < p \le M$ (the case $p=2M-2$ is easy, since $f$ is a positive definite polynomial, so $\sigma$ is a minimizer and hence so are tight designs). Transferring the problem back to the sphere $\mathbb S_{\mathbb F}^{d-1}$, as explained in Section \ref{subsec:antip}, finishes the proof of part \eqref{t2} of Theorem \ref{t:mainsphere}.
%This plays an important role in what follows and it is precisely functions with this property of alternating derivative sign (for a large enough index) which our results apply to.
\end{proof}

In what follows we give a proof of Theorem \ref{thm:tight}, splitting it into two separate cases, depending on whether the design $ \mathcal C $ contains two points separated by the diameter of $ \Omega $; equivalently, depending on the parity of the strength $ M $ of $\mathcal C$.

\begin{proposition}
    \label{prop:2m}
    Theorem~\ref{thm:tight} holds when $ M = 2m $, $ m \geq 1 $.
\end{proposition}
\begin{proof} 
    Let $t_1 < \ldots < t_m < t_{m+1}= 1$ be the values of $ \tau(x,y) = \cos(\theta(x,y)) $ occurring in $ \mathcal C $. Let further
\[
    g_k(t) = \prod_{i=1}^k (t - t_i), \qquad 1 \leq  k \leq m+1.
\]

and
\begin{equation}\label{eq:HermiteGEven}
g(t) = g_m(t)\, g_{m+1}(t) =  (t-1) g^2_m(t).
\end{equation}
To prove the statement of the theorem, we verify the following chain of inequalities, satisfied for arbitrary $ \mu \in \mathcal P(\Omega) $, similar to the proof of Lemma \ref{lem:lp}, 
\begin{equation}
    \label{eq:inequalities}
    I_f(\mu) \geq  I_{H[f,g]}(\mu) \geq I_{H[f,g]}(\sigma) = I_{H[f,g]}(\mu_{\mathcal C}) = I_f(\mu_{\mathcal C}).
\end{equation}

Since $ g(t) \leq 0 $ for $ t \in [-1,1] $, Lemma~\ref{lem:remaind} implies that $ f(t) \geq H[f,g](t) $, $ t \in [-1,1] $, which gives the first inequality. The equality $ I_{H[f,g]}(\sigma) = I_{H[f,g]}(\mu_{\mathcal C}) $ is satisfied  since $ \mathcal C $ is a design of strength $ 2m \geq \deg H[f,g] $. The last equality holds since the interpolant $ H[f,g] $ agrees with $ f $ at the cosines of distances occurring in $ \mathcal C $.   All that remains to show is  the second inequality: by Proposition \ref{prop:pdmin}, it will follow from the positive definiteness of $ H[f,g] $, which we will now demonstrate.

%%% The equality $ I_{H[f,g]}(\sigma) = I_{H[f,g]}(\mu_{\mathcal C}) $ follows since $ \mathcal C $ is a design of strength $ 2m \geq \deg H[f,g] $. The last equality holds since $ H[f,g] $ agrees with $ f $ at the cosines of distances occurring in $ \mathcal C $. Since $ g(t) \leq 0 $ for $ t \in [-1,1] $, Lemma~\ref{lem:remaind} implies that $ f(t) \geq H[f,g](t) $, $ t \in [-1,1] $, which gives the first inequality. It remains to show the second inequality: it will follow from the positive definiteness of $ H[f,g] $, which we will now demonstrate.

For any $n < m$, the degree of $ g_{m+1} (t)  Q_n(t)$ is at most $2m$. As  $ \mathcal C $ is a $ 2m $-design, for every fixed $y \in \mathcal C$ there holds

\begin{equation*}
    \begin{aligned}
    \int_{-1}^1 g_{m+1}(t)  Q_n(t) d\nu^{(\alpha, \beta)} 
    & = \int_{\Omega} g_{m+1}( \tau(x,y) )  Q_n( \tau(x,y) ) d\sigma(x) \\
    & = \frac1{|\mathcal C|} \sum_{x \in \mathcal C} g_{m+1}( \tau(x,y) )  Q_n( \tau(x,y) ) \\
    & = \frac1{|\mathcal C|} \sum_{i=1}^{m+1} c_i g_{m+1}( t_i)  Q_n(t_i) = 0,
    \end{aligned}
\end{equation*}
since, by construction, $ g_{m+1} $ is annihilated on all the $ t_i $. 
The constants $ c_i $ are given by, for any fixed $y \in \mathcal{C}$,
$$ c_i = |\{ x \in \mathcal C\ |\  \tau(x,y) = t_i \}|.$$
Both $ g_{m+1} $ and $ Q_{m+1} $ are monic, so we conclude that
\[
    g_{m+1}(t) =  Q_{m+1}(t) + \gamma Q_m(t),
\]
for some $ \gamma \in \mathbb R $. By Proposition~\ref{prop:subproducts}, subproducts of zeros of $ g_{m+1} $, which we denote by $ g_k $, $ 1 \leq k \leq m $, can be expressed as linear combinations of $ Q_n $ with positive coefficients, and therefore are positive definite.

According to the Newton's formula \eqref{eq:newton_formula}, the Hermite interpolant of $ f $ can be expressed as the sum of partial products of factors of $ g $ multiplied by the appropriate divided difference. We will use this formula to show that $ H[f,g] $ is positive definite. Indeed,  \eqref{eq:newton_formula} gives

\begin{equation}\label{eq:NewtonEven}
    H\left[f,g\right](t) = f(t_1) + \sum_{k=1}^{m}  \bigg(g_{k}(t) g_{k-1}(t)\, Q\left[f, g_{k}g_{k-1}\right]\left(t_{k}\right) + g^2_{k}(t) \, Q\left[f, g_{k}^2\right]\left(t_{k+1}\right)\bigg),
\end{equation}
where as usual, $ g_0 = 1 $. Observe that the divided differences in the last equation are nonnegative due to \eqref{eq:div_difference}, as the function $ f $ is absolutely monotonic of degree $ 2m $. Since we have shown that each $ g_k $ is positive definite, Schur's theorem implies that so are $g_k^2$ and $g_k g_{k-1}$, and  it follows that $ H[f,g]$ is positive definite as well.

\end{proof}

Before turning to the proof of Theorem~\ref{thm:tight} for tight designs of odd strength, recall the definition of the adjacent polynomials $ Q^{1,1}_n = Q_n^{(\alpha + 1, \beta+1)}$ for $n \geq 0$. They are monic, orthogonal with respect to the measure 
$$ d\nu^{({\alpha+1},{\beta+1})}(t) =  \frac1{\gamma_{\alpha+1,\beta+1}} (1-t)^{\alpha+1} (1+t)^{\beta+1} dt =  \frac{\gamma_{\alpha,\beta}}{\gamma_{\alpha+1,\beta+1}} (1-t^2) d\nu^{({\alpha},{\beta})}(t),$$ since the polynomials $ Q_n^{(\alpha, \beta)}(t) $ are orthogonal with respect to measure $ d\nu^{({\alpha},{\beta})} $.% (for simplicity we disregard the normalization here).

\begin{proposition}
    \label{prop:2m-1}
    Theorem~\ref{thm:tight} holds when $ M = 2m-1 $, $ m\geq 1 $.
\end{proposition} 
\begin{proof} 
    %One can verify that a symmetric $2t$ spherical design is a $(2t+1)$ design. 
    Suppose that $\mathcal C \subset \Omega$ is a tight $(2m-1)$-design.  As discussed  in  Section ~\ref{s:des} tight designs  of odd strength necessarily contain antipodal points, i.e. there exist $x$, $y\in \mathcal C$ such that $ \vartheta (x,y) =\pi$ and thus $-1\in \mathcal A (\mathcal C) = \{ \tau (x,y) |\, x,y \in \mathcal C \}$. %Further, tight odd spherical designs must be centrally symmetric.   
    Let $-1 = t_1 < \ldots < t_m < t_{m+1}= 1$ be the values of $ \tau(\theta(x,y)) $ for  $ x,y\in \mathcal C $, and set

$$ w(t) = \prod_{j=2}^{m} (t - t_j)$$
% $$ g_1(t) = (t+1) f(t)^2.$$
and
\begin{equation}\label{eq:HermiteGOdd}
 g(t) = w^2(t)(t^2-1).
 \end{equation}
As in the proof of Proposition~\ref{prop:2m}, we need to verify the
inequalities \eqref{eq:inequalities}. Applying Lemma~\ref{lem:remaind} to $
H[f,g] $ gives the first inequality; it remains to show positive-definiteness
of $ H[f,g] $.
 
For $n < m-1$, the degree of $(1-t^2) w(t) Q^{1,1}_n(t)$ is at most $2m-1$, so for any $y \in \mathcal C$ there holds

\begin{equation*}
    \begin{aligned}
 \frac{\gamma_{\alpha+1,\beta+1}}{\gamma_{\alpha,\beta}}   \int_{-1}^1 w(t)  Q^{1,1}_n(t) d\nu^{(\alpha+1, \beta+1)} 
    & = \int_{\Omega} (1- \tau^2(x,y))w(\tau(x,y))  Q^{1,1}_n(\tau(x,y)) d\sigma(x) \\
    & = \frac1{|\mathcal C|} \sum_{x \in \mathcal C} (1- \tau^2(x,y))w(\tau(x,y))  Q^{1,1}_n(\tau(x,y)) \\
    & = \frac1{|\mathcal C|} \sum_{j=1}^{m+1} c_j (1- t_j^2)w( t_j)  Q^{1,1}_n(t_j)  = 0,
    \end{aligned}
\end{equation*}
as $ (1- t^2)w( t) $ is annihilated on the cosines of distances from $ \mathcal C $.  
Because $w(t)$ is a degree $m-1$ monic polynomial, the above implies $w(t) = Q^{1,1}_{m-1}(t)$. By Proposition~\ref{prop:subproducts}, this also means that for $ 2 \leq k \leq m-1$, polynomials $\prod_{j=2}^{k} (t-t_j)$ are linear combinations of $Q^{1,1}_n$ with nonnegative coefficients. Since the cone of functions with nonnegative Jacobi coefficients with respect to $ Q^{1,1}_{n} $ is closed under multiplication, polynomials $\prod_{j=2}^{k} (t-t_j)^2$ and  $ (t-t_k)\prod_{j=2}^{k-1} (t-t_j)^2 $ also have nonnegative Jacobi coefficients in $  Q^{1,1}_{n}  $. Due to Lemma~\ref{lem:Krein}, since $t-t_1 = t+1$, we obtain that 
\begin{equation}
    \label{eq:partialprods}
    a_k(t) := (t-t_1)(t-t_l)\prod_{j=2}^{k-1} (t-t_j)^2 \quad\text{and}\quad  b_k(t):=  (t-t_1)\prod_{j=2}^{k} (t-t_j)^2 ,
\end{equation}
are linear combinations of $ Q_n^{(\alpha, \beta)} $ with positive coefficients, that is, they are positive definite on $\Omega$ for $ 1\leq k \leq m $.

We conclude by the same observations as in the proof of Proposition~\ref{prop:2m}; in particular, the positive definiteness of the Hermite interpolant $ H[f,g] $ follows from the representation 
\begin{equation}\label{eq:NewtonOdd}
    H\left[f,g\right](t) = f(t_1) + b_1(t)\, Q\left[f, b_1\right]\left(t_{2}\right) + \sum_{k=2}^{m}  \bigg( a_k(t)\, Q\left[f, a_k\right]\left(t_{k}\right) + b_{k}(t) \, Q\left[f, b_{k}\right]\left(t_{k+1}\right) \bigg),
\end{equation}
combined with the absolute monotonicity of $ f $ to degree $ 2m-1 $, which implies positivity of the divided differences $Q$. \end{proof}

\begin{example}
As an example of another application of Theorem \ref{thm:tight}, consider the case that $f(t)=a+bt+ct^2+dt^3$ is given as potential function. In this case, some elementary considerations show that if
\begin{enumerate}
\item[(i)] $d\leq 0$,
\item[(ii)] $c\geq -3d$,
\item[(iii)] $c^2-3bd\geq 0$, 
\item[(iv)] $-c-\sqrt{c^2-3bd}\leq 3d$, and,
\item[(v)] $-c+\sqrt{c^2-3bd}\geq 3d$,
\end{enumerate}
then $f$ is absolutely monotonic of degree $2$ up to a constant. Hence, any potential function of the above form has as minimizer of the $f$-energy on any of the projective spaces a tight $2$-design. In particular, for $f$ as above, the icosahedron is a minimizer of energy integral $I_{f}(\mu)$ over symmetric measures on the sphere $\S^2$. Note that the constant term can be ignored, so it suffices to only consider the sign of derivatives. In particular, if $b>0$ and $d$ becomes sufficiently small in magnitude, the above inequalities will hold.

For comparison, on $\S^2$, $f(\tau(x,y))=f(2|\langle x,y \rangle|^2-1)$ is positive definite (up to a constant), precisely when $\widehat{f}_1,\widehat{f}_2,$ and $\widehat{f}_3$ are positive, or equivalently (by calculation),
\begin{enumerate}
\item[(i)] $4b-2c+3d\geq 0$,
\item[(ii)] $2c-d\geq 0$, and,
\item[(iii)] $d\geq 0$.
\end{enumerate}
Thus, $I_f$ is minimized by the surface measure $\sigma$   precisely for the  coefficients satisfying the above inequalities.

\end{example}

\subsection{Uniqueness of minimizers supported on tight designs}\label{sec:uniq}
The proofs in the last section left the question of uniqueness of minimizers open. Are there any other minimizers for $p$-frame energies when tight designs minimize and $p$ is not an even integer? The answer, as this section details, is no. 

In general, whenever a tight design minimizes $I_f$ for some kernel $f$ that is strictly absolutely monotonic of degree $M$ and which satisfies $f^{(M+1)}(t)< 0$, $t\in(-1,1)$, the energy is minimized only by a tight design, although such designs are not necessarily unique up to equivalence, as mentioned in Section \ref{s:des}. Before stating our result in full, we introduce  a couple standard lemmas (in slightly simplified form adapted to our needs).  

Let $N_M$ denote the cardinality of a tight $M$-design in $\Omega$ or, more precisely, the linear programming upper bound on the cardinality of $M$-designs \cite{delsarteSpherical1977,hoggar1982t}, which  is  well-defined even if tight $M$-designs do not exist and coincides with their. In fact, tight designs are often equivalently defines in terms of  this quantity. 

The first lemma, which can be found in \cite{levenshteinUniversal1998}, states that tight designs have the smallest cardinality among all {\em{weighted}} designs of given strength.

\begin{lemma}\label{thm:LowerBoundsDesigns}
Let $(\mathcal{B}, w)$ be a weighted $M$-design in $\Omega$. %Let $K$ be the cardinality of a tight $M$-design over $\Omega$. 
Then $|\mathcal{B}|\geq N_M$ and equality holds if and only if $w(x) = \frac{1}{|\mathcal{B}|}$ for all $x \in \mathcal{B}$ and $\mathcal{B}$ is a tight $M$-design.
\end{lemma}

The second lemma shows that tight designs have the largest cardinality among all sets with a given number of distinct distances. 

%\begin{lemma}\label{lem:DGSLinProg}
%Let $\mathcal{C}\subset \Omega$ be an $M$-design. Let $K$ be the cardinality of a tight $M$-design in $\Omega$. Then $|\mathcal{C}|\geq K$, and  equality holds if and only if the distance set $\mathcal{A}(\mathcal{C})$ is exactly the set of roots of a polynomial with roots given as the distances occurring in a tight $M$ design in $\Omega$.

\begin{lemma}\label{lem:DGSLinProg}
Let $\mathcal{B}\subset \Omega$ be an $m$-distance set, i.e. $|\mathcal A (\mathcal B) | = m$.  Then $| \mathcal B | \leq N_{2m}$. Moreover, if $\mathcal B$ is antipodal (contains a pair of points diameter apart), then $| \mathcal B | \leq N_{2m-1}$.
\end{lemma}

This lemma was proved in \cite{delsarteSpherical1977}  for the sphere and in \cite{hoggar1982t} for projective spaces. We are now ready for the uniqueness result. 

%These lemmas found in \cite{levenshteinUniversal1998} and \cite{delsarteSpherical1977,hoggar1982t} respectively let us relate the degree, strength, and cardinality of designs, and provide us with information about the distance set $\mathcal{A}(\mathcal{C})$ of tight designs.

\begin{theorem}\label{t:uniq} 
Suppose that a tight $M$-design $\mathcal{C}$ minimizes the $f$-energy integral, for $f$ strictly absolutely monotonic of degree $M$ and such that $f^{(M+1)}(t)< 0$, $t\in(-1,1)$. Then any minimizer of $I_{f}$ must be a tight $M$-design.
\end{theorem}
\begin{proof}
The argument developed to prove Theorem~\ref{thm:tight} may be described concisely through the following string of inequalities
$$I_{f}(\mu) \geq I_{H[f,g]}(\mu)\geq I_{H[f,g]}(\sigma)=I_{H[f,g]}(\mu_{\mathcal{C}})=I_{f}(\mu_{\mathcal{C}}),$$
where $g$ is of the form \eqref{eq:HermiteGEven} or \eqref{eq:HermiteGOdd}, as is appropriate.
In order for $I_{f}(\mu)=I_{f}(\mu_{\mathcal{C}})$ to hold, the inequalities must be equalities. The first inequality can only be an equality in the case that $\mathcal{A}(\operatorname{supp}(\mu)) \subseteq \mathcal{A}(\mathcal{C})$. This follows from the fact that $H[f,g](t)<f(t)$ for all $t \not\in \mathcal{A}(\mathcal{C})$ by the remainder formula from Lemma \ref{lem:remaind}. In particular, this shows that $|\operatorname{supp}(\mu)|$ is finite.  Moreover,  Lemma \ref{lem:DGSLinProg}  then guarantees that $| \operatorname{supp}(\mu)| \leq N_M = | \mathcal C|,$ since $N_M$ is  increasing with $M$.

 Now assume that the second inequality above  is an equality. %only when $\mu$ is a weighted design of strength at least equal to that of the minimizing tight design $\mathcal{C}$. 
 We first note that since $f$ is strictly absolutely monotonic of degree $M$, $f(t_1) \geq 0$, and the divided differences appearing in \eqref{eq:NewtonEven} or \eqref{eq:NewtonOdd} are all positive due to \eqref{eq:div_difference}. Thus, $H[f,g]$ is a linear combination (possibly modulo a constant), with positive coefficients, of positive definite polynomials of degrees $ 1, ..., M$, so $H[f,g] = a_0 + \sum_{j=1}^{M} a_j C_j$, where $a_j > 0$ for $j > 1$ and $a_0 \geq 0$. We see that $\mu$ must then be a weighted $M$-design, and due to Lemma \ref{thm:LowerBoundsDesigns}, %and the fact that $\mathcal{C}$ is a tight $M$-design,  
 we have $|\operatorname{supp}(\mu)| \geq  N_M  = | \mathcal{C}| $.
 
 Therefore, $|\operatorname{supp}(\mu)| =  N_M  = | \mathcal{C}| $, and the second part of Lemma \ref{thm:LowerBoundsDesigns} implies that $\operatorname{supp}(\mu)$ is a tight $M$-design and $\mu$ has equal weights.
%
% Lemma \ref{lem:DGSLinProg}  then tell us that this is only possible if  
%$$|\mathcal{A}(\operatorname{supp}(\mu))| \geq |\mathcal{A}(\mathcal{C})|.$$
%Thus, the distance sets must be the same, and therefore so are the cardinalities of the sets, making $\operatorname{supp}(\mu)$ a tight $M$-design. Since $\mu$ is a weighted $M$-design, it must be a tight $M$-design.
 \end{proof}

\section{Optimality of the 600-cell}
\label{sec:600cell} 
This section concerns only the $p$-frame kernels; it will be shown here that the 600-cell minimizes the $ p $-frame energy on $ \S^3 $ for a certain range of $ p $. The $600$-cell is one of the six $4$-dimensional convex regular polytopes; it has $600$ tetrahedral faces, which explains the origin of its name. When its 120 vertices are identified with unit quaternions, they give a representation of the elements of a group known as the binary icosahedral group \cite{stillwell2001story}. 

As discussed above  \eqref{eq:symmetrized_pkernel}, optimization of $ p $-frame energy on the sphere $ \S^3 $ is equivalent to optimization of the expression $ \iint_{(\RP^3)^2} f(\tau(x,y))\, d\mu(x) d\mu(y) $ over measures $ \mu $ on $ \RP^3  $, where the kernel $ f $ is given by 
\begin{equation*}
    f(t)=\bigg(\frac{1+t}{2}\bigg)^{\frac{p}{2}}.
\end{equation*}
We therefore assume for the rest of this section the underlying space to be $ \RP^3 $, and use the corresponding Jacobi polynomials $ C_n^{(-1/2, 1/2)}(t) $.
Following the approach of the previous section, we will establish a sequence of inequalities similar to \eqref{eq:inequalities}. 

The $600$-cell is only a projective $5$-design and therefore not tight. The authors in \cite{cohnUni2007}, motivated by an approach found in the paper \cite{Andreev1999}, found means to prove universal optimality of the $600$-cell by using a higher degree interpolating polynomial. The $600$-cell has the notable property that $7$th, $8$th, and $9$th degree harmonic averages over it vanish, although the $6$th degree average does not. This allows for constructing a degree $8$ polynomial $h$ which is less than or equal to $f$, positive definite, and agrees with $f$ at the distances appearing in the $600$-cell, and which finally has the property that its $6$th Jacobi coefficient vanishes.

For a polynomial $ h $ of the form, 
\begin{equation}
    \label{eq:h_poly}
    h = \sum_{\substack{n=0\\n\neq 6}}^8\widehat h_n\, C_n^{\left(1/2,-1/2\right)}(t),
\end{equation}
the coefficients $ \widehat h_n $ can be uniquely determined as functions of $ p $ by setting
\[ 
    \begin{aligned}
        h(t_i) =& f(t_i), \qquad 1 \leq i \leq 5\\
        h'(t_i) =& f'(t_i), \qquad 2 \leq i \leq 4,
\end{aligned}
\] 
where $ -1 =  t_1 < t_2 < \ldots < t_5 =1 $ are the values of $ \tau(x,y) $ when vectors $ x, y $ vary over the vertices of the 600-cell,  see the proof of Theorem~\ref{thm:600} below.
It turns out that for all $p\in[8,10]$, $ \widehat h_n(p) \geq 0 $ when $ 0 \leq n \leq 8, n \neq 6 $. We apply a computer-assisted approach to verify this positivity; specifically, using interval arithmetic, we compute values of $\widehat h_n(p) $ on a grid fine enough to guarantee that $ \widehat h_n(p) \geq 0 $. The details of this computation are available in the auxiliary files of the arXiv submission of this paper. Even though the computations performed are carried out in finite floating point precision, interval arithmetic guarantees that the results of these computations lie in precisely defined intervals (using libraries \cite{InriaForge,zimmermannComputational2018,sageMath}). The computer-assisted argument yields the following.

\begin{lemma}\label{lemma:600} If $ p \in [8,10] $ and the polynomial $ h $ is constructed as above, the coefficients $\widehat{h}_{n}$ in the Jacobi expansion \eqref{eq:h_poly}  satisfy $ \widehat h_n(p) \geq 0 $.
\end{lemma}

Using this fact we show optimality of the $600$-cell on the range $p\in [8,10]$. 

\begin{theorem}\label{thm:600} The $600$-cell minimizes the $p$-frame energy for $p\in[8,10]$ over Borel probability measures on $\S^{3}$ or $\RP^{3}$.  
\end{theorem}
\begin{proof} 
    Let $f(t) = \left( \frac{t + 1}{2} \right)^{p/2}$ for some $8 < p < 10$, $t_1 = -1$, $t_2 = \frac{-\sqrt{5} - 1}{4}$, $t_3 = - \frac{1}{2}$, $t_4 = \frac{\sqrt{5}-1}{4}$, and $t_5 = 1$. Let $h(t)$ be the $8${th} degree polynomial given by \eqref{eq:h_poly}, such that $h(t_i) = p(t_i)$ for $1 \leq i \leq 5$, and $h'(t_i) = p'(t_i)$ for $2 \leq i \leq 4$. By Lemma \ref{lemma:600}, the coefficients $ \widehat h_n $ are non-negative for $p\in[8,10]$.

 Let $p(t) =  (t^2 - 1) \prod_{i = 2}^{4} (t - t_i)^2$ and $\tilde{h}(t) = H[f,p](t)$. Then we also have $\tilde{h}(t) = H[h,p](t)$. This gives
$$ f(t) - \tilde{h}(t) = \frac{f^{(8)} (\xi)}{8!} p(t) \geq 0,$$
and
$$ h(t) - \tilde{h}(t) = \frac{h^{(8)}(\nu)}{8!} p(t) \leq 0.$$
We thus have $ f(t)- h(t) = f(t) - \tilde{h}(t) + \tilde{h}(t)- h(t) \geq 0$.  
Since $h(t)$ is positive definite and $\widehat{h}_6 = 0$, for the 600-cell $ \mathcal C_{600} $, we have the following sequence of inequalities % turns into equalities
\[
    I_f(\mu) \geq  I_{h}(\mu) \geq I_{h}(\sigma) = I_{h}(\mu_{\mathcal C_{600}}) = I_f(\mu_{\mathcal C_{600}}),
\]
implying that equally weighted vertices of $ \mathcal C_{600} $ minimize $ p $-frame energy.
\end{proof}

\section{Conjectured minimizers} \label{sec:ea1} 
\subsection{New small weighted projective design} 
\label{subsec:new_quadrature}
We now collect facts on the $85$ vector system which was found while numerically minimizing the $p=5$ frame potential in $\C^5$. This system of vectors forms a weighted design of strength $3$, or equivalently, for the functional $\sum_{i,j}|\langle v_{i},v_{j}\rangle|^6\omega_{i}\omega_{j}$, the weighted system takes the value $1/35$, thus minimizing this quantity over all probability measures $\mu=\sum_{i}\delta_{v_{i}}\omega_{i}$, $\sum_{i} \omega_{i}=1$ supported on unit vectors $\|v_{i}\|=1$ in $\C^5$ \cite{Welch1974}. The above construction appears to be new especially when comparing its size to previously obtained bounds from \cite{lyubich2013rec} for smallest known $3$ weighted designs in $\C^5$. 

One part of the system is well studied, given by the root vectors corresponding to the $45$ 2-reflections which generate the unitary reflection group $W(K_{5})$ of $51840$ elements \cite{lehrer2009unitary}. This group is alternatively described as the group $G_3(10)\simeq (C_6 \times SU_4(2)):C_2$, one of the maximal finite irreducible subgroups of $GL_{10}(\mathbb{Z})$ \cite{souvignier1994irreducible}. $SU_4(2)$ here is just the special linear group of $4\times 4$ matrices, unitary matrices over $\F_{2^2}$, with determinant one.

Choosing the representation of the root vectors in $W(K_{5})$ as $X_1=\{\sigma((1,0,0,0,0))\}\cup\{\sigma(\frac{1}{2}(0,1,\pm \omega,\pm \omega, \pm 1))\}$ under cyclic coordinate permutations, $\sigma$, the new weighted design arises when this system is joined with some other 40 vectors. The second system may be described as $\Psi=\{\sigma(\frac{1}{\sqrt{3}}(1,0,\pm \omega, \pm \omega,0))\}\cup \{\sigma(\frac{1}{\sqrt{3}}(1,\pm \omega,\pm 1,0,0))\}$ also generated under cyclic coordinate permutations. The projective design is finally given by assigning weights to the $W(K_{5})$ system joined with the $40$ vector system after giving $\Psi$ the orientation $X_2=U\Psi$, where 
\begin{equation} U=\frac{1}{2}\left[ \begin{array}{ccccc}
1 & -\omega & -\omega & 1 & 0 \\ 
-1 & 1 & -\omega^2 & 0 & -\omega^2 \\ 
\omega^2 & 0 & -\omega^2 & 1 & 1 \\ 
0 & 1 & \omega & -\omega & -1 \\ 
\omega^2 & \omega & 0 & -\omega & \omega^2 
\end{array} \right], 
\end{equation}
is unitary ($\omega=e^{2\pi i/3}$). With the above orientation the $40$ points in $X_2$ appear to fit so that each point is a maximizer of the projective distance from each of the $45$ vectors in the $W(K_5)$ system and vice versa. If so, the additional $40$ points satisfy that they are the points at greatest distance from the original $45$, in particular.

To form a weighted $3$-design, the corresponding weights for $X_1$, the $45$ vector system, are $\omega_1=\frac{4}{315}$, and for the remaining $40$ vectors in $X_2$, the weights are $\omega_2=\frac{3}{280}$. In total the distribution of absolute values of inner products that appears in the unweighted $85$ vector system is given in Table \ref{tab:85-vect}. The supplementary files in the arXiv version of this manuscript provide a magma script which verifies that $\sum_{i,j}|\langle v_{i},v_{j}\rangle|^6\omega_{i}\omega_{j}=1/35$, so that the system is a projective $3$-design. This script can additionally be used to show the automorphism group of the above system of $85$ vectors is isomorphic to the group $SU_{4}(2)$ of order $|SU_{4}(2)|=25920=2^6\cdot(2^4-1)\cdot(2^3+1)\cdot(2^2-1)$, through use of a library from \cite{hughesSpherical}.

\begin{table}
    \caption{Table of inner products between vectors in parts $X_1,X_2$ of the new cubature formula of $85$-vectors. $N$ counts the number of times a value occurs as an entry in $|X_{i}'X_{j}|$, $i,j=1,2$.}
    \begin{tabular}{c | c | c } 
 & $|\langle x,y\rangle|$ & $N$ \\ \hline
$|X_1'X_1|$ & $0,1/2,1$ & $540,1440,45$ \\
$|X_2'X_2|$ & $1/3,1/\sqrt{3},1$ & $1080,480,40$ \\
$|X_1'X_2|$ & $0,1/\sqrt{3}$ & $720,1080$ \\
$|X_2'X_1|$ & $0,1/\sqrt{3}$ & $720,1080$ \\
    \end{tabular}

\label{tab:85-vect}
\end{table}

The above construction hides the relation between its two parts. The $85$ vectors in $\C^5$ may be seen, after canonically embedding the vectors in $\R^{10}$, as the weighted union of vectors coming from two $10$ dimensional lattices. Under this identification, the $45$ vectors in the $W(K_5)$ system may be selected as, up to projective equivalence (modulo multiples of sixth roots of unity), the $270$ minimal vectors of the lattice called $(C_6 \times SU_4(2)):C_2$ in the database \cite{nebesloane}, and the other $40$ points are taken one from each antipodal pair of the $80$ minimal vectors of the shorter Coxeter-Todd lattice, $O_{10}$ detailed in \cite{rains1998shadow}. The relationship between these two lattices is that $(C_6 \times SU_4(2)):C_2$ is similar to the maximal even sub-lattice of $O_{10}$. In our tables, we choose to name these the $W(K_5)$ and $O_{10}$ lattices. We prefer an alternative name for the first since the automorphism group of each lattice is $(C_6 \times SU_4(2)):C_2$. 

Altogether, upon splitting the weights across minimal vectors in appropriately scaled and oriented copies of these lattices and then complexifying everything, one arrives at the cubature formula, which when viewed projectively, is a system of $85$ vectors improving on the best previous known bound of size $320$ for such a formula (see \cite{shatalovIsometric2001}). Some experiments suggest this might be the smallest sized weighted projective $3$-design in $\CP^4$. Expecting that this code might be optimal in a few other settings, we conjecture:

\begin{conjecture} The code constructed in this section of $85$ points in $\C^5$ is universally optimal.
\end{conjecture}

 This is an example of one of the `highly symmetric tight frames', as was later demonstrated in \cite{MW19}.

\subsection{Other weighted designs}
\label{subsec:other_designs}

\subsubsection{\texorpdfstring{$11$ points in $\R^{3}$}{11 points in R3}}
\label{subsection:3-11}

It seems that as $p$ goes to $6$ from below, the limiting minimizing configuration on the sphere $\S^{2}$ is of the following form. Concisely, the system consists of all combinations of signs of the $6$ vectors below, 

$$ \left[ \begin{array}{ccc} 1 & 0 & 0 \\ 0 & 1 & 0 \\ 0 & 0 & 1 \\ \frac{2}{\sqrt{7}} & \sqrt{\frac{3}{7}} & 0 \\ \frac{2}{\sqrt{7}} & 0 & \sqrt{\frac{3}{7}} \\ \sqrt{\frac{1}{7}} & \sqrt{\frac{3}{7}} & \sqrt{\frac{3}{7}} \end{array} \right]$$ 
with the weights, 

$$\frac{2}{27}, \frac{1}{10}, \frac{1}{10}, \frac{49}{540}, \frac{49}{540},\frac{49}{540} $$ 
on each line. The off-diagonal inner products are then $$1/7,-1/7,5/7,-5/7,\sqrt{3/7},-\sqrt{3/7},0,\sqrt{1/7},-\sqrt{1/7},4/7,-4/7,\sqrt{4/7},-\sqrt{4/7}$$ appearing in number, $(10,18,10,10,14,10,14,6,2,4,4,6,2)$ respectively. From these  facts, one may check that the $11$ lines defined by these vectors forms a projective $3$-design. Notably, this is the same extremal code, which forms a minimal cubature formula and is found also in \cite[page 135]{Rez}.

\subsubsection{\texorpdfstring{$16$ points in $\R^{3}$}{16 points in R3}}
\label{subsection:3-16}

Lines through antipodal points in the union of a regular icosahedron with its dual dodecahedron. The frequencies of absolute values of inner products are $N(\sqrt{\frac{1}{15}(5-2\sqrt{5})})=60,$  $N(\frac{\sqrt{75+30\sqrt{5}}}{15})=60,$  $ N(\frac{1}{3})=60,$  $ N(\frac{1}{\sqrt{5}})=30,$  $ N(\sqrt{\frac{5}{9}})=30,$  and   $ N(1)=60$. The weights making this configuration a projective $4$-design are $\omega_1=5/84$ and $\omega_2=9/140$ for the icosahedron and dodecahedron vertices respectively.

\subsubsection{\texorpdfstring{$11$ points in $\R^{4}$}{11 points in R4}}
\label{subsection:4-11}

See Table~\ref{tab:4a11gram} for what appears to be the limiting minimizing configuration as $p$ goes to $6$ from below when minimizing over $\S^3$.

\begin{table}
\caption{The Gram matrix of the weighted projective $2$-design in $\R\mathbb{P}^3$ which appears as a minimizer as $p\rightarrow 4^{-}$ along with ordered weights, with each weight corresponding to the vector with inner products given in the adjacent row. In the matrix, $a$ and $b$ are $\frac{\sqrt{5}+1}{6}$ and $\frac{1}{6}\sqrt{(6-2\sqrt{5})}$, respectively.}  
$\displaystyle
\left[ \begin{array}{ccccccccccc|c}
                  \rule{0pt}{3ex}  \ 1 &      -\frac{2}{3} &          a &          a &          a &          a &          b &          b &          b &          b &  \frac{\sqrt{6}}{6} & \frac{3}{40} \\

     \rule{0pt}{3ex} -\frac{2}{3} &         \ 1 &         -b &         -b &         -b &         -b &         -a &         -a &         -a &         -a &  \frac{\sqrt{6}}{6} & \frac{3}{40} \\

       \rule{0pt}{3ex}  \ a &        -b &          1 &        \ \frac{1}{3} &      \ \frac{1}{3} &        -\frac{1}{3} & -\frac{\sqrt{2}}{3} & -\frac{\sqrt{2}}{3} &  \frac{\sqrt{2}}{3} &  \frac{\sqrt{2}}{3} &  \ \frac{\sqrt{6}}{6} & \frac{3}{32} \\

       \rule{0pt}{3ex}  \ a &        -b &        \ \frac{1}{3} &          \ 1 &        -\frac{1}{3} &       \ \frac{1}{3} &  \frac{\sqrt{2}}{3} & -\frac{\sqrt{2}}{3} &  \frac{\sqrt{2}}{3} & -\frac{\sqrt{2}}{3} &  \ \frac{\sqrt{6}}{6} & \frac{3}{32} \\

       \rule{0pt}{3ex}  \ a &        -b &       \ \frac{1}{3} &        -\frac{1}{3} &          \ 1 &        \ \frac{1}{3} &  -\frac{\sqrt{2}}{3} &  \frac{\sqrt{2}}{3} & -\frac{\sqrt{2}}{3} & \frac{\sqrt{2}}{3} &  \ \frac{\sqrt{6}}{6} & \frac{3}{32} \\

       \rule{0pt}{3ex}  \ a &        -b &       -\frac{1}{3} &       \ \frac{1}{3} &        \ \frac{1}{3} &          \ 1 & \frac{\sqrt{2}}{3} &  \frac{\sqrt{2}}{3} & -\frac{\sqrt{2}}{3} &  -\frac{\sqrt{2}}{3} &  \ \frac{\sqrt{6}}{6} & \frac{3}{32} \\

       \rule{0pt}{3ex}  \ b &        -a & -\frac{\sqrt{2}}{3} &  \frac{\sqrt{2}}{3} &  -\frac{\sqrt{2}}{3} & \frac{\sqrt{2}}{3} &          \ 1 &        \ \frac{1}{3} &        \ \frac{1}{3} &       -\frac{1}{3} & -\frac{\sqrt{6}}{6} & \frac{3}{32} \\

       \rule{0pt}{3ex}  \ b &        -a & -\frac{\sqrt{2}}{3} & -\frac{\sqrt{2}}{3} &  \frac{\sqrt{2}}{3} &  \frac{\sqrt{2}}{3} &        \ \frac{1}{3} &          \ 1 &       -\frac{1}{3} &        \ \frac{1}{3} & -\frac{\sqrt{6}}{6} & \frac{3}{32} \\

       \rule{0pt}{3ex}  \ b &        -a &  \frac{\sqrt{2}}{3} &  \frac{\sqrt{2}}{3} & -\frac{\sqrt{2}}{3} & -\frac{\sqrt{2}}{3} &        \ \frac{1}{3} &       -\frac{1}{3} &          \ 1 &        \ \frac{1}{3} & -\frac{\sqrt{6}}{6} & \frac{3}{32} \\

       \rule{0pt}{3ex}  \ b &        -a &  \frac{\sqrt{2}}{3} & -\frac{\sqrt{2}}{3} & \frac{\sqrt{2}}{3} &  -\frac{\sqrt{2}}{3} &       -\frac{1}{3} &        \ \frac{1}{3} &        \ \frac{1}{3} &          \ 1 & -\frac{\sqrt{6}}{6} & \frac{3}{32} \\

\rule{0pt}{3ex}  \frac{\sqrt{6}}{6} & \frac{\sqrt{6}}{6} &  \frac{\sqrt{6}}{6} &  \frac{\sqrt{6}}{6} &  \frac{\sqrt{6}}{6} &  \frac{\sqrt{6}}{6} & -\frac{\sqrt{6}}{6} & -\frac{\sqrt{6}}{6} & -\frac{\sqrt{6}}{6} & -\frac{\sqrt{6}}{6} &          \ 1 & \frac{1}{10} \\
                
 \end{array}\right]$
\label{tab:4a11gram}
\end{table}

\subsubsection{\texorpdfstring{$24$ points in $\R^{4}$}{24 points in R4}}
\label{subsection:4-24}

The regular $24$ cell, or alternatively the $D_{4}$ root system. The frequencies of absolute values of inner products are $N(0)=216,\ N(\frac{1}{\sqrt{2}})=144,\ N(\frac{1}{2})=192,\ \text{ and }\ N(1)=24$. The configuration is unweighted as a projective $3$-design.

\subsubsection{\texorpdfstring{$16$ points in $\R^{5}$}{16 points in R5}}
\label{subsection:5-16}

Lines through antipodal points in the following construction. Take all permutations of $\pm\frac{1}{\sqrt{30}}(-5,1,1,1,1,1,1)$ and $\frac{1}{\sqrt{6}}(1,1,1,-1,-1,-1)$ and consider these as vectors in the copy of $\S^{4}$ in $\S^{5}$ on the plane perpendicular to $(1,1,1,1,1,1)$. The frequencies of absolute values of inner products are $N(\frac{1}{3})=90,\ N(\frac{1}{5})=30,\ N(\frac{1}{\sqrt{5}})=120,\ \text{ and } N(1)=16$. The weights making this a projective $2$-design are $\omega_1=\frac{5}{84}$ and $\omega_2=\frac{9}{140}$ for the above parts respectively.

\subsubsection{\texorpdfstring{$41$ points in $\R^{5}$}{41 points in R5}}
\label{subsection:5-41}

An example of a design construction appearing in \cite{stroud1967some}. The configuration comprises of lines through antipodal points in the following construction. Let $A$ be the set of vectors which are permutations of $(\pm 1,0,0,0,0)$, $B$ permutations of $(\pm \sqrt{\frac{1}{2}},\pm \sqrt{\frac{1}{2}},0,0,0)$, and $C$ permutations of $(\pm\sqrt{\frac{1}{5}},\pm\sqrt{\frac{1}{5}},\pm\sqrt{\frac{1}{5}},\pm\sqrt{\frac{1}{5}},\pm\sqrt{\frac{1}{5}})$. The frequencies of absolute values of inner products are $N(0)=600,\ N(\frac{1}{5})=160,\ N(\frac{3}{5})=80,\ N(\sqrt{\frac{1}{5}},\ N(\sqrt{\frac{2}{5}})=320,\ \text{ and }\ N(1)=41$. The weights making this a projective $3$-design are $\omega_1=\frac{2}{105}$, $\omega_2=\frac{8}{315}$, and $\omega_{3}=\frac{25}{1008}$, on $A,B$, and $C$ respectively.

\subsubsection{\texorpdfstring{$22$ points in $\R^{6}$}{22 points in R6}}
\label{subsection:6-22}
Lines through antipodal points in a hemicube/cross polytope compound, where the hemicube is within the cube dual to the cross polytope. The frequencies of absolute values of inner products are $N(0)=30,\ N(\frac{1}{\sqrt{6}})=192,\ N(\frac{1}{3})=240,\ \text{ and }\ N(1)=22$. The weights making this a projective $2$-design are $\omega_1=3/64$ on the hemicube and $\omega_2=1/24$ on the cross-polytope.

\subsubsection{\texorpdfstring{$63$ points in $\R^{6}$}{63 points in R6}}
\label{subsection:6-63}
Lines through antipodal points in the union of minimal vectors of $E_{6}$ and its dual lattice, $E_{6}^{*}$. The frequencies of absolute values of inner products are $N(0)=1620,\ N(\frac{1}{4})=432,\ N(\frac{1}{2})=990,\ N(\sqrt{\frac{3}{8}})=864,\ \text{ and }\ N(1)=63$. The weights making this a projective $3$-design are $\omega_1=1/60$ and $\omega_2=2/135$ on the minimal vectors of $E_6$ and its dual, respectively.

\subsubsection{\texorpdfstring{$91$ points in $\R^{7}$}{91 points in R7}}
\label{subsection:7-91}
The configuration is projectively composed of the union of the minimal vectors of $E_{7}$ and its dual lattice, $E_{7}^{*}$. The frequencies of absolute values of inner products are $N(0)=3906,\ N(\frac{1}{27})=756,\ N(\frac{1}{8})=2016,\ N(\frac{\sqrt{3}}{9})=1512,\ \text{ and }\ N(1)=91$. The weights making this a projective $3$-design are $\omega_1=8/693$ and $\omega_2=3/308$ on the $E_7$ part and its dual, respectively. The cubature formula appears also in \cite{Nozaki12}.

\subsubsection{\texorpdfstring{$36$ points in $\R^{8}$}{36 points in R{8}}}
\label{subsection:8-36}
The edge midpoints of a regular simplex. The frequencies of absolute values of inner products are $N(\frac{2}{7})=756,\ N(\frac{5}{14})=504,\ \text{ and }\ N(1)=36$. This code is a projective $1$-design with equal weights.

\subsubsection{\texorpdfstring{$21$ points in $\C^{3}$}{21 points in C{3}}}
\label{subsection:3-21}
A structured union of a maximal (tight) simplex (equiangular tight frame, or ETF) of $9$ vectors and $4$ mutually unbiased bases (a $4$-MUB) of $12$ vectors. The frequencies of absolute values of inner products are $N(0)=96,\ N(\frac{1}{2})=72,\ N(\frac{1}{\sqrt{3}})=108,\ N(\frac{1}{\sqrt{2}})=144,\ N(1)=21$. The weights making this a projective $3$-design are $\omega_1=4/90$ on the $9$-ETF and $\omega_2=\frac{1}{20}$ on the $4$-MUB.  

\section{\texorpdfstring{$p$}{p}-frame energies in non-compact spaces}\label{sec:non}

In the previous sections, we used linear programs to bound energies on compact two-point homogeneous spaces. This approach can be extended to $p$-frame energies in non-compact spaces as well. Just as above, we consider $\F=\R, \C,$ or $\mathbb H$. In this setting, we consider  the set of probability measures $\P(\F^d)$ with the additional restriction

\begin{equation}\label{eqn:non-compact} \int_{\F^d} |x|^2 d\mu(x)=1 
\end{equation}
for each $\mu\in\P(\F^d)$. This normalization allows us to obtain a direct extension of above results for the spherical case, and by scaling, solutions to more general problems can be obtained from these results. A similar problem of finding maximizers for $p$-frame energies for $p\leq 2$, subject to the condition that measures be isotropic, was investigated in \cite{glazyrin2019moments}.

For a potential function$f= f( \tau(x,y)) = f(2 | \langle x, y \rangle|^2 -1)$ we define the energy with respect to measure $\mu\in\P(\F^d):$

$$I_f(\mu) =\int_{\F^d} \int_{\F^d} f(\tau(x,y)) d\mu(x)d\mu(y).$$

We will be concerned in this section only with the case that $f(\tau(x,y)) = |\langle x,y \rangle|^p$. The Jacobi polynomials for the projective spaces $\FP^{d-1}$, as above, are denoted $C_m$.

\begin{lemma}
	\label{lem:non-compact}
For $p\geq 2$, assume $f(t) = \left( \frac {t+1} {2}\right)^{\frac p 2}\geq h(t) = \sum\limits_{m=0}^{\infty} \widehat{h}_m C_m(t)$ for all $t\in [-1,1]$, where $\widehat{h}_m \geq 0$ for all $m\geq 0$. Then $I_f(\mu) \geq \widehat{h}_0$ for all $\mu\in\P(\F^d)$ satisfying \eqref{eqn:non-compact}.
\end{lemma}

\begin{proof}
Since discrete masses are weak-$*$ dense in $\P(\F^d)$, it is sufficient to prove the inequality for them only. Let $\mu$ take the form $\mu=\frac{1}{N}\sum\limits_{i=1}^N \delta_{x_i}$, $x_i\in\F^{d}$ and set $y_i=\frac {x_i} {|x_i|}$. (Note, if $x_i$ is $0$ then we can assign an arbitrary unit vector for $y_i$). Then, 

$$I_f(\mu) =\frac 1 {N^2} \sum\limits_{i,j=1}^N |\langle x_i, x_j\rangle |^p = \frac 1 {N^2} \sum\limits_{i,j=1}^N |x_i|^p |x_j|^p |\langle y_i, y_j \rangle |^p  = \frac 1 {N^2} \sum\limits_{i,j=1}^N |x_i|^p |x_j|^p f( \tau( y_i, y_j))$$

$$ \geq \frac 1 {N^2} \sum\limits_{i,j=1}^N |x_i|^p |x_j|^p h(\tau(y_i,y_j)) = \frac {1}{N^2} \sum\limits_{m=0}^{\infty} \widehat{h}_m \sum\limits_{i,j=1}^N |x_i|^p |x_j|^p C_m (\tau(y_i,y_j)).$$

For any $m\geq 1$, $C_m$ is positive definite on $\FP^{d-1}$, so each sum $\sum\limits_{i,j=1}^N |x_i|^p |x_j|^p C_m(\tau(y_i,y_j))$ is non-negative. Thus,
$$I_f(\mu) \geq \widehat{h}_0 \frac 1 {N^2} \sum\limits_{i,j=1}^N |x_i|^p |x_j|^p C_0(\tau(y_i,y_j))= \widehat{h}_0 \left(\frac 1 N \sum\limits_{i=1}^N |x_i|^p \right)^2.$$

Since $p\geq 2$, 
$$\frac 1 N \sum\limits_{i=1}^N |x_i|^p \geq \left( \frac 1 N \sum\limits_{i=1}^N |x_i|^2 \right)^{\frac p 2},$$ 
holds by Jensen's inequality. The constraint \ref{eqn:non-compact} is equivalent to $\frac 1 N \sum\limits_{i=1}^N |x_i|^2=1$, and so combining all inequalities, we complete the proof of the lemma. \end{proof} 
Lemma \ref{lem:non-compact} gives that any linear programming bounds for $p$-frame energies applicable to the spherical/projective case will work in the non-compact setting as well. As a consequence of this approach we obtain the following result.

\begin{theorem}
	\label{thm:non-compact}
Let $ \mathcal C $ be a set of arbitrary unit representatives of a tight projective $M$-design, $M\geq 2$, in $\FP^{d-1}$ and $f(\tau(x,y))=|\langle x, y\rangle |^p$ with $p\in[2M-2,2M]$.  Then 
$$ \mu_{\mathcal C} = \frac{1}{|\mathcal C|} \sum_{x \in \mathcal C} \delta_{x} $$
is a minimizer of
$$ I_{f}(\mu) = \int_{\F^d} \int_{\F^d} f( \tau(x,y) ) d \mu(x) d \mu(y)$$
over the set of probability measures on $\F^d$ satisfying the constraint \ref{eqn:non-compact}.
\end{theorem}

\begin{proof}
For the proof, we take $f(t)= \left( \frac {t+1} {2}\right)^{\frac p 2}$, $h$ to be the interpolating polynomial $H[f,g]$ used in the proof of Theorem~\ref{thm:tight} or $h$ used in the proof of Theorem \ref{thm:600}, and $h^*(x,y) = |x|^p |y|^p h( \tau( \frac{x}{|x|}, \frac{y}{|y|}))$ for all $x,y \in \mathbb{F}^d $. We follow the same line of reasoning as before to find
\begin{equation}
I_f(\mu) \geq  I_{h^*}(\mu) \geq I_{h^*}(\sigma^*) = I_{h^*}(\mu_{\mathcal C}) = I_f(\mu_{\mathcal C}),
\end{equation}
where $\sigma^*$ is the uniform probability measure on the unit sphere in $\mathbb{F}^d$ (and so projects to the Haar measure on $\mathbb{FP}^{d-1}$).

All inequalities are verified in a similar manner as in the previous section, except for $I_{h^*}(\mu) \geq I_{h^*}(\sigma^*)$. This part follows from Lemma \ref{lem:non-compact} applied to $h^*$ because $I_{h^*}(\sigma^*) = I_h(\sigma)$ is precisely $\widehat{h}_0$ for positive definite functions $h$.
\end{proof}

{\it Note:} A similar result may be derived in the same manner as above for $\mathcal{C}$, a set of arbitrary unit representatives of the $600$-cell in $\mathbb{RP}^3$ and $p \in [8,10]$, in light of Theorem \ref{thm:600}.

\section{\texorpdfstring{Mixed volume inequalities}{Mixed volume inequalities}}
\label{sec:mixed_volume}

In this section we demonstrate an intriguing connection between the $p$-frame energy and convex geometry. We begin by briefly recalling some of the basic notions from convex geometry. See \cite[Ch. 2]{koldobsky} for a more thorough development. 

\indent Let $K$ be a convex body and $\sigma_{K}(u)$ be the surface measure of $K$, that is, a measure supported on the unit sphere $\S^{d-1}$, satisfying 
$$\sigma_{K}(B)=| \{x\in\partial K,\ \text{the outer unit  normal to }K\ \text{at}\ x\ \text{belongs to }B \}|_{d-1}$$
for all Borel sets $B\subset\S^{d-1}$, where $|\cdot|_{d-1}$ denotes the  $(d-1)$-dimensional  Hausdorff measure. For example, if $K$ is a polytope with faces $\{K_i\}_{i=1}^m$ and normals $\{{ n}_{i}\}_{i=1}^m$, $\sigma_{K}$ is atomic with mass $|K_{i}|_{d-1}$ at each ${ n}_i$,
$$\sigma_{K} =\sum\limits_{i=1}^m |K_{i}|_{d-1}\delta_{{ n}_i},$$
and if $K= \mathbb B$ is the $d$-dimensional unit ball, then $\sigma_K$  simply coincides with the standard (unnormalized) uniform surface area measure  $\sigma_K (B) =  |B|_{d-1} = \frac{2\pi^{d/2}}{\Gamma (d/2)} \sigma (B)$. 

\indent Recall that for a convex body, $K\subset\R^{d}$, the {\em{support function}} $h_{K}(u)$ of $K$ takes the form 
$$ h_{k}(u)=\sup\limits_{v\in K} \langle u,v \rangle.$$ Given two convex bodies $K$ and $L$, and $p\geq 1$, define 
$$V_p(K,L)=\frac{p}{d}\lim\limits_{\epsilon\rightarrow 0} \frac{|K+_{p} \epsilon L|-|K|}{\epsilon},$$
where $ K+_{p} \epsilon L$ is the convex body with support function $h_{K+_{p} \epsilon L}(u)$ satisfying 
$$h_{K+_{p} \epsilon L}(u)^p=h_{K}(u)^p+\epsilon h_{L}(u)^p.$$ 
 Note that for $L=\mathbb{B}_d$ is the unit ball and $p=1$, the above quantity is just the definition of the surface area of $K$. In general, $V_{p}(K,L)$ is known as the {\em{$L_p$-mixed volume of $K$ and $L$}}. The following alternative integral representation for $V_{p}(K,L)$ is known
$$V_{p}(K,L)=\frac{1}{d}\int_{\S^{d-1}}h_{L}(u)^{p} d\sigma_{K}^p(u),$$
where $d\sigma_K^p(u)=h_{K}(u)^{1-p}d\sigma_K(u)$, so that in particular $d\sigma_K^1(u)=d\sigma_K(u)$ . \\

\indent Now, call a probability measure $\mu$ supported on $\S^{d-1}$ admissible, if it is symmetric and not concentrated on a subspace. A  classical result which follows from Minkowski's theorem, says that any admissible measure can be realized as the surface area measure of a symmetric convex body; see more in \cite[Ch. 7]{schneider}.

\indent The {\it projection body} $\Pi K$ of a convex body $K$ is defined to be a body such that for each $u\in\S^{d-1}$  
$$h_{\Pi K}(u)=\left|K|u^{\perp} \right|_{d-1},$$ 
that is, the support function of $\Pi K$  equals  the volume of the projection of $K$ onto the hyperplane orthogonal to  $u$ \cite{bourgain}. Since 
$$ \left|K|u^{\perp} \right|_{d-1}=\frac{1}{2}\int_{\S^{d-1}} |\langle u,v \rangle| d\sigma_{K}(v), $$ 
the identities
\begin{align*}
I_{|t|}(\sigma_K) & =\int_{\S^{d-1}}\int_{\S^{d-1}}|\langle u,v \rangle|\, d\sigma_{K}(u)d\sigma_{K}(v)  =  2  \int_{\S^{d-1}}   \left|K|u^{\perp} \right|_{d-1}  d\sigma_{K}(u)\\
& =    2  \int_{\S^{d-1}}   h_{\Pi K}(u)  \,  d\sigma_{K}(u)   =  2d  \, V_1 ( K, \Pi K )
\end{align*}
finally establish the connection between $L_1$-mixed volumes and $1$-frame energies. 

Our main theorem, Theorem  \ref{t:mainsphere}, shows that all minimizers of $I_{|t|^{p}}(\mu)$ over probability measures are admissible when a corresponding tight design exists, as this measure is both discrete and can be taken to be symmetric. From this, we obtain what appears to be a new observation, namely the following:
\begin{proposition}
The minimum  of the quantity  $$\frac{V_1(K,\Pi K)}{|\partial K|^2}$$ over all symmetric convex bodies in $\R^d$ is achieved when  $K$ is a cube.
\end{proposition}    
Indeed, it is easy to see that, when $K$ is a cube, the surface measure $\sigma_K$ is equally distributed on the vertices of a cross-polytope, which minimizes the $p$-frame energy for $p=1$. 

One may also define $L^p$-intersection bodies $\Pi_p K $ \cite{Lutwak,Lutwak2}  in a similar fashion 
and obtain analogous relations for other values of $p$. Doing so allows one to infer similar  statements  for $V_p(K,\Pi_p K)/|\partial K|^2$ for the several dimensions and ranges of $p$ considered in this manuscript (for which tight designs exist), as well as pose conjectures corresponding to the numerically obtained minimizers. We anticipate, in particular, in accordance with Conjecture \ref{conj:discrete},  that whenever  $p$ is not an even integer,  this quantity is always minimized by a convex body which is polyhedral (with discrete surface measure).

\section{Other linear programming applications} 
\label{sec:other_kernels}

\subsection{Other energy problems}
We now discuss some  problems  related to minimization of $p$-frame energies and other energies with degree $M$ absolutely monotonic potentials. For $p\not\in 2\mathbb N$, the potential functions $f(t)= \frac1{2^{p/2}} \cdot (1+t)^{p/2}$, corresponding to the $p$-frame energy,    have the property that not only their derivatives switch signs for  large enough orders, but also the coefficients in their  Jacobi expansion have alternating signs. While the proof of optimality depended heavily on the former, we look into the latter property now. 

In a sense, the most natural polynomial potential functions to consider when approximating $f(t)$ are of the form 
\begin{equation}\label{eq:gfunc}
g(t)=\sum\limits_{m=0}^{k} \widehat{f}_m C_{m}(t)-\beta C_{k+1}(t),\end{equation}
where $\widehat{f}_m, \beta \geq 0$. Because any earlier truncation of function $f$ is positive definite, so that $I_{f}(\mu)$ is minimized by surface measure, the point at which the first negative coefficient comes in is the first interesting truncation to consider. At first glance, it may seem that minimizers of $I_{g}(\mu)$ might act like those  of $I_f (\mu)$. When $\beta$ is too large however, this cannot be true, for in this case a single Dirac mass $\nu=\delta_{x}$ gives 
$$I_{g}(\nu)=g(\tau(x,x))=\sum\limits_{m=0}^k \widehat{f}_m-\beta,$$
which can be smaller than the value obtained on any other measure. Instead of looking  at the energy with potential  $g$, we shall  considering the limiting problem of constrained optimization
\begin{equation}\label{eq-constr}  \max\limits_{\mu\in\mathcal{P}(\Omega)} \int_{\Omega}\int_{\Omega}  C_{k+1}(\tau(x,y)) d\mu(x)d\mu(y)\ 
\text{s.t.}\  \int_{\Omega}\int_{\Omega} C_{j}(\tau(x,y)) d\mu(x)d\mu(y)=0,\ j=1,...,k, 
\end{equation}
which is again minimized by tight designs, as the below argument shows.

\begin{theorem} If a tight  $k$-design $\mathcal{C}$ in $\Omega$ exists, then the  measure $\mu_{\mathcal C}$ which distributes mass evenly among the points of $\mathcal C$ solves the optimization problem \eqref{eq-constr}   over probability measures. 

\begin{proof} Set $f = C_{k+1}$. If $k=2m$, we construct the polynomial $h$ by applying Hermite interpolation to  $f$ at $t=1$,  and to $f$ and $f'$ at the other $m$ values of $\mathcal A (\mathcal C) = \{ \tau(x,y)| \, x,y \in \mathcal C, x\neq y\}$, so that $\textup{deg}\, h \le 2m =k$. When $k=2m-1$, i.e. $\mathcal C$ contains antipodal pairs, we apply interpolation of order $1$ at $t=\pm 1$ and of order $2$ at the other $m-1$ values of  $\mathcal A (\mathcal C)$, resulting in $\textup{deg}\, h \le 1+2(m-1) =2m-1 =k$. The remainder formula \eqref{eq:div_difference} then gives that the difference
$$ f(t) - h(t)  = \frac{f^{(k+1)} (\xi)}{(k+1)!} (t-1)  \cdot \begin{cases} \prod\limits_{\alpha \in \mathcal A (\mathcal C)\setminus \{ 1\} } (t-\alpha)^2  \\ \prod\limits_{\alpha \in \mathcal A (\mathcal C)\setminus \{ \pm 1\} } (t-\alpha)^2 \cdot (t+1) \end{cases} $$ is non-positive for $t\in [-1,1]$. This holds because $f^{(k+1)} (\xi) >0$, as it is simply the leading coefficient of $f=C_{k+1}$.  Since $h$ is a polynomial of degree $k$, the constraints in \eqref{eq-constr} imply that for any admissible $\mu \in \mathcal P (\Omega)$, $I_h (\mu) = \widehat{h}_0 = I_h (\sigma).$ We therefore obtain  
$$I_{f}(\mu)\leq I_{h}(\mu)=I_{h}(\sigma)=I_{h}(\mu_{\mathcal{C}})=I_{f}(\mu_{\mathcal{C}}),$$
where the penultimate equality relies on the fact that $\mathcal C$ is a $k$-design, and the last one follows from interpolation. \end{proof}

\end{theorem}

Note that the argument for uniqueness applies for the above problem, much  as it did in the case of  degree $M$ absolutely monotonic functions $f$. The difference lies in the fact  that here the design condition arises from the constraints. 

It is possible to show that generally,  both for  problem \eqref{eq-constr} above and the problem of minimizing $g$-energy for $g$ as in \eqref{eq:gfunc}, there exist finitely supported minimizing measures. The reader familiar with the geometry of moment constrained probability measures will not be surprised by this since the extreme points of sets of such measures are finitely supported. The complete argument for existence of discrete measures can be found in \cite{bilyk2019energy}, but a quick sketch is as follows. Each problem above has solutions which may be found by re-writing the problem as maximization of a convex functional over the set of  moment constrained probability measures (the second problem \eqref{eq-constr} is already exactly in such form) and thus the maximum should be attained on an extreme point, i.e. a discrete measure. In both cases, in addition to the existence of a discrete minimizer,  one can obtain quantitative upper bounds on the support size in terms of the number of constraints.

\subsection{Causal variational principle}\label{sec:causal}

Define the kernel 
\begin{equation}\label{eq:CausalVar}
F(t) = F_{\tau} (t) := \max \{ 0, 2\tau^2  \big(1+t \big) \big( 2 - \tau^2 (1-t) \big) \}.
\end{equation}
for $\tau>0$. The  minimization problem for the energy 
\begin{equation}\label{eq:CausalVarEn}
I_{F}(\mu) = \int_{\mathbb{S}^2} \int_{\mathbb{S}^2} F (\langle x,y \rangle )  d\mu(x)  d\mu (y)
\end{equation}
is known as the {\em{causal variational principle}} on the sphere and is connected to relativistic quantum field theory.  It is conjectured in \cite{finsterSupport2013} that there exist discrete minimizers for $\tau \ge 1$ and, moreover, that all the minimizers of \eqref{eq:CausalVarEn}
are discrete whenever  $\tau >\sqrt{2}$. The background on this problem can be found in \cite{finsterSupport2013, BFSM}.

Here we confirm this conjecture for two values of $\tau >0$, for which we can show that the cross-polytope or orthoplex and the icosahedron indeed minimize the energy. These  minimizing measures were suggested by numerical experiments in \cite{finsterSupport2013}.

\subsubsection{Cross-polytope: \texorpdfstring{$\tau = \sqrt{2}$}{tau = sqrt(2)}} When $\tau = \sqrt{2}$,  we have  $$ \mathcal L (t) = \max \{ 0, 8t^2 +8t \},$$ and thus  $\mathcal L (0) = 0$. 
Define the measure $$\nu = \frac16 \sum_{i=1}^3 \bigg( \delta_{e_i} + \delta_{-e_i} \bigg),$$ where $\{e_1, e_2, e_3\}$ is an orthonormal basis of $\mathbb R^3$, i.e. $d\nu$ is a measure whose mass is equally concentrated in the vertices of a cross-polytope. Then we have, 

\begin{proposition}\label{p:tau1} The measure $\nu$ is a minimizer for the energy, $I_{\mathcal{L}}$, over probability measures on $\S^2$ for $\tau = \sqrt{2}$.
\end{proposition}

\subsubsection{Icosahedron\texorpdfstring{: $\displaystyle{ \tau^2 = \frac{2\sqrt{5}}{\sqrt{5} -1 }}$}{}\,} This value of $\tau$ is chosen so 
that $\mathcal L_\tau (1/\sqrt{5}) = 0$.  
Let $\mathcal C\subset \S^2$ be the vertices of a regular icosahedron and let 
$$\displaystyle{\nu = \frac1{12} \sum_{x\in \mathcal C}  \delta_x}$$
be the uniform measure on the vertices of the icosahedron. 

\begin{proposition}\label{p:tau2} The measure $\nu$ is a minimizer for the energy, $I_{\mathcal{L}}$, over probability measures on $\S^2$ for $\tau^2= \frac{2\sqrt{5}}{\sqrt{5} -1}$.
\end{proposition}

The proof of Propositions \ref{p:tau1} and \ref{p:tau2} is another application of the linear programming framework, which in this case is particularly straightforward, since, unlike the previous sections,  a single auxiliary function must be constructed to certify the solution. We postpone the details to the Appendix, see Section~\ref{ap:causal}.

\section{Further remarks}
\label{sec:remarks}
We have many remaining questions about the $p$-frame energies, and many curiosities were brought to our attention through our numerical study. One immediate question concerns uniqueness of the $600$-cell as a minimizer for $\RP^3$ and $p\in(8,10)$, which we expect to hold. In Section \ref{sec:uniq}, we mentioned that tight designs, generally, are not unique (not even up to unitary equivalence). This is known to be true in particular for SIC-POVM's in $\C^3$ through the characterization of all SIC-POVM's in this dimension in \cite{szollosi}. It is interesting whether it is more often the case that infinite families arise or that such configurations are isolated, as is known to be the case when $d=2$ \cite{zau99}.

An interesting observation is that some configurations minimize $p$-frame energies for a range of $p$ (the $600$-cell for example), while others, like the $p=3$ minimizer in $\RP^{7}$, do not minimize on an entire range between even integers. When minimizers have the same support for a range $p\in(2k-2,2k)$, it indicates that the supporting configuration has to be a weighted $k$-design.

For the $36$ points in $\RP^7$ given as the midpoints of edges of a regular simplex, one can check that the strength of this configuration as a design is too small to satisfy the above condition. Further, the value of the energy for a measure which equally distributes over this set when compared against the surface measure is too large to be a minimizer when $p$ is close to (but less) than four.

We do not expect that the minimizers of $ p $-frame energies are necessarily weighted $k$-designs, but noticed that many of the configurations which showed up numerically as limit points of the even $p$ values (from below) were smallest known weighted designs. Informally, one might expect for these configurations to have isolated or small support since if the points become too well distributed, the distribution gets ``closer'' to surface measure which means the averages of the configuration over the negative coefficient terms in the Jacobi expansion of $f$ vanish. Since one wants to maximize such contributions, the vectors might be taken close to a weighted $k$-design, but just ``barely'' so.

Some other cases where the support of a minimizer appears to change within even arguments are for $\RP^2$, $p\in(4,6)$ and $\RP^3$, $p\in(2,4)$. One might be tempted to suggest that the configurations which show up as minimizers on an interval are universally optimal, but this is not the case. For example, the $D_4$ root system, which appears to be optimal on $p\in(4,6)$ for $\S^{3}$, is not universally optimal \cite{ccek07}. Nonetheless, in the limited numerical experiments which were run outside of the parameters found in Tables \ref{table:real} and \ref{table:complex}, it appears uncommon that a configuration be optimal on a range of $p$, and when it does happen, the configuration is highly structured.

This suggests another phenomenon similar to the notion of universal optimality, and we are tempted to conjecture that in the real case for $d>2$ there are only finitely  many configurations which optimize the $p$-frame energy on a whole range of $p\in [2k-2,2k]$.

As was mentioned earlier, we conjecture that all energies with $p$ not an even integer have discrete and only discrete minimizers. It is known already that $p$-frame energies among other non-positive definite functions cannot contain an open set in the interior of the support of a minimizing measure \cite{bilyk2019energy}. We are not familiar with an argument which  would rule out the possibility of an arc of a circle being contained in the support of a minimizing measure for $p$-frame potentials, or non-positive definite truncations of their Jacobi polynomial expansions.

Looking at the tables, one can note that as the value of $p$ increases, for $p$ not even, the support size of a candidate appears to be monotonically increasing. Further, for a fixed dimension, the support size seems to grow polynomially in $p$. We do not have an explanation for this phenomenon.

One motivating reason for considering the other projective spaces beyond the real case, is the connection with the problem known as Zauner's conjecture, on existence of tight projective $2$-designs, these being best known by their alternative name SIC-POVMs. The moment constrained problem considered in Section~\ref{sec:other_kernels} for $k=2$ has the property that a discrete solution with support size bounded by an explicit function of $d$ exists regardless of whether a tight design exists. Further, the minimizer must be a (weighted) projective $2$-design. If a SIC-POVM exists, it must solve this problem or the $p$-frame energy problem for $p\in(2,4)$. 

Interestingly, it is conjectured \cite{CKM} from numerical evidence that the property analogous to Zauner's should not hold in the quaternionic setting. If this is true, it is curious what instead should appear as a minimizer.

Finally, we give additional details on how we made the conjectures found in Tables \ref{table:real} and \ref{table:complex}. The numerical method employed to find conjectured minimizers involved two steps. Early on, we used a quick first order gradient descent method to minimize energies. Afterwards we implemented an arbitrary precision library with a second order method to check our conjectures and test endpoint behavior. 

\section{Acknowledgements}

We express our gratitude to the  following organizations that  hosted subsets of the authors during the work on this paper: AIM, ICERM, INI, CIEM, Georgia Tech. The authors were supported by the following grants  DMS-1665007, DMS-2054606 (DB),  DMS-2054536 (AG), the Graduate Fellowship 00039202 (RM),  and  in part by the grant DMS-1600693 and Tripods grant CCF-1934904 (JP), all from the US National Science Foundation, as well as the Simons collaboration grant for mathematicians (DB) and the AMS-Simons travel grant (OV). This work was also supported by EPSRC grant no EP/K032208/1 (the authors' visit to INI) and NSF DMS-1439786 (research in groups at ICERM).

We are extremely grateful to Galyna Livshyts for bringing to our attention the alternative interpretation of the $p$-frame energy, as given in Section~\ref{sec:mixed_volume}.  We thank Greg Blekherman, Henry Cohn, David de Laat, Michael Lacey, Bruce Reznick, Alexander Reznikov, Daniela Schiefeneder, and Yao Yao for helpful discussions.

\section{Appendix}
\begingroup
\fontsize{9pt}{12pt}\selectfont
\subsection{Parameters of the conjectured and rigorous optimizing configurations} 
Table \ref{table:opt_parameters} gives weights and inner products of the support vectors of $ p $-frame minimizing measures in $ \R^d $, see Table~\ref{table:real} and Section~\ref{subsec:other_designs}.

\begin{table}[h!]
    \caption{Optimal and conjectured configurations for $p$-frame energies. The configurations are supported on $N$ unit norm vectors in $\R^d$, and are strength $M$ designs.   Only the absolute values of inner products are given.\\  
    {\footnotesize Note: $\alpha=\frac{\sqrt{75+30+\sqrt{5}}}{15},\beta=\sqrt{\frac{1}{15}(5-2\sqrt{5})},\text{ and } \gamma=\frac{\sqrt{6-2\sqrt{5}}}{6}$.}}
    \label{table:opt_parameters}
    \begin{tabular}{cccccc} 
        $d$ & $N$ & $M$& Weights & Inner Products & Name \\ \hline
        \rule{0pt}{3ex}$2$ & $N$ & $N-1$ & ${1}/{N}$ & $|\cos (2\pi j/N)|, 1\leq j < N $  & $2N$-gon\\ 
        $d$ & $d$ & $1$ & ${1}/{N}$ & $0$ & cross polytope\\ 
        $3$ & $6$ & $2$ & ${1}/{N}$ & $\frac{1}{\sqrt{5}}$ & icosahedron\\ 
        $3$ & $11$ & $3$ & $\frac{1}{10},\frac{2}{27},\frac{49}{540}$ & $0,\frac{1}{7},\frac{4}{7},\frac{5}{7},\sqrt{\frac{1}{7}},\sqrt{\frac{3}{7}},\sqrt{\frac{4}{7}}$ & Reznick weighted design\\ 
        $3$ & $16$ & $4$ & $\frac{5}{84},\frac{9}{140}$ & $\frac{1}{3},\frac{1}{\sqrt{5}},\sqrt{\frac{5}{9}},\alpha,\beta$ & icosahedron and dodecahedron\\ 
        $4$ & $11$ & $2$ & $\frac{1}{10},\frac{3}{32},\frac{3}{40}$ & $\frac{1}{3},\frac{\sqrt{2}}{3},\frac{\sqrt{6}}{6},\frac{\sqrt{5}+1}{6},\gamma$ & weighted design\\ 
        $4$ & $24$ & $3$ & ${1}/{N}$ & $0,{1}/{2},{1}/{\sqrt{2}}$ & $D_4$ roots\\ 
        $4$ & $60$ & $5$ & $ {1}/{N}$ & $0,\frac{\sqrt{5}\pm 1}{4},\frac{1}{2}$ & $600$-cell\\ 
        $5$ & $16$ & $2$ & $\frac{5}{84},\frac{9}{140}$ & $\frac{1}{5},\frac{1}{3},\frac{1}{\sqrt{5}}$ & hemicube\\ 
        $5$ & $41$ & $3$ & $\frac{25}{1008},\frac{8}{315},\frac{2}{105}$ & $0,\frac{1}{5},\frac{1}{2},\frac{3}{5},\frac{1}{\sqrt{2}},\frac{1}{\sqrt{5}},\sqrt{\frac{2}{5}}$ & Stroud weighted design\\ 
        $6$ & $22$ & $2$ & $\frac{3}{64},\frac{1}{24}$ & $0,\frac{1}{3},\frac{1}{\sqrt{6}}$ & cross polytope and hemicube \\ 
        $6$ & $63$ & $3$ & $\frac{2}{135},\frac{1}{60}$ & $0,\frac{1}{4},\frac{1}{2},\sqrt{\frac{3}{8}}$ & $E_6$ and $E_6^{*}$ roots\\ 
        $7$ & $28$ & $2$ & ${1}/{N}$ & $1/3$ & kissing $ E_8 $\\ 
        $7$ & $91$ & $3$ & $\frac{3}{308},\frac{8}{693}$ & $0,\frac{1}{27},\frac{1}{8},\frac{\sqrt{3}}{9}$ & $E_7$ and $E_7^{*}$ roots\\ 
        $8$ & $36$ & $1$ & ${1}/{N}$ & $ {5}/{14},{2}/{7}$ & mid-edges of regular simplex\\ 
        $8$ & $120$ & $3$ & ${1}/{N}$ & $0,1/2$ & $E_8$ roots\\ 
        $23$ & $276$ & $2$ & ${1}/{N}$ & $1/5$ & equiangular\\ 
        $23$ & $2300$ & $3$ & ${1}/{N}$ & $0,1/3$ & kissing $\Lambda_{24}$\\ 
        $24$ & $98280$ & $5$ & ${1}/{N}$ & $0,1/4,1/2$ &  $\Lambda_{24}$ minimal vectors\\ 
    \end{tabular} 
\end{table} 

\subsection{Numerical LP bounds}
Table~\ref{table:numlp} collects linear programming lower bounds corresponding to small values of $d$ and odd values $p$ for the $p$-frame energy on $\S^{d-1}$.
\begin{center}
\begin{table}[h]
\caption{Numeric linear programming lower bounds for odd-valued $p$-frame energies.}
\begin{tabular}{cccc}
\label{table:numlp}
$d $ & $ p=3 $  & $ p=5 $ &  $p=7$ \\  \hline
$3 $ & $ 0.2412 $ & $ 0.1655 $ & $ 0.1248 $ \\
$4 $ & $ 0.1612 $ & $ 0.09607 $ & $ 0.06454 $ \\
$5 $ & $ 0.1170 $ & $ 0.06169 $ & $ 0.03740 $ \\
$6 $ & $ 0.08970 $ & $ 0.04240 $ & $ 0.02344 $ \\
$7 $ & $ 0.07142 $ & $ 0.03060 $ & $ 0.01556 $ \\
$8 $ & $ 0.05852 $ & $ 0.02291 $ & $ 0.01080 $ \\
$9 $ & $ 0.04902 $ & $ 0.01770 $ & $ 0.007768 $ \\
$10 $ & $ 0.04180 $ & $ 0.01401 $ & $ 0.005750 $ \\
$11 $ & $ 0.03616 $ & $ 0.01131 $ & $ 0.004360 $ \\
$12 $ & $ 0.03166 $ & $ 0.009290 $ & $ 0.003375 $ \\
$13 $ & $ 0.02801 $ & $ 0.007737 $ & $ 0.002658 $ \\
$14 $ & $ 0.02499 $ & $ 0.006524 $ & $ 0.002125 $ \\
$15 $ & $ 0.02248 $ & $ 0.005561 $ & $ 0.001721 $ \\ 
$16 $ & $ 0.02035 $ & $ 0.004785 $ & $ 0.001413 $ \\ 
$17 $ & $ 0.01853 $ & $ 0.004152 $ & $ 0.001171 $ \\
$18 $ & $ 0.01696 $ & $ 0.003630 $ & $ 0.0009813 $ \\ 
$19 $ & $ 0.01559 $ & $ 0.003195 $ & $ 0.0008280 $ \\
$20 $ & $ 0.01440 $ & $ 0.002830 $ & $ 0.0007054 $ \\
$21 $ & $ 0.01335 $ & $ 0.002520 $ & $ 0.0006047 $ \\
$22 $ & $ 0.01242 $ & $ 0.002256 $ & $ 0.0005217 $ \\ 
$23 $ & $ 0.01159 $ & $ 0.002028 $ & $ 0.0004529 $ \\
$24 $ & $ 0.01085 $ & $ 0.001832 $ & $ 0.0003952 $ \\

\end{tabular}
\end{table}
\end{center}

\vspace{-5mm}

            \subsection{Causal variational principle}\label{ap:causal}

            \subsubsection{Cross-polytope}

            Let the following polynomial be given, $$ H(t) = 8 t^2 + 8 t.$$
            It is easy to see that $H$ is positive definite on $\mathbb S^2$. Additionally, it is obvious that  
            $$ H(t) \le \mathcal L (t) \qquad \textup{ for all  } t \in [-1,1],$$
            and 
            $$ H(-1) = \mathcal L (-1) = 0 , \quad H(0) = \mathcal L (0) = 0, \quad H(1)  = \mathcal L (1) = 16, $$ 
            so that $H$ coincides with $\mathcal L$ on the set $\{ x\cdot y: x,y \in \supp \nu \}$. 

            We obtain that for any measure $\mu \in \mathcal P$,
            \begin{equation}
                I_{\mathcal L} (\mu ) \ge I_H (\mu ) \ge I_H (\sigma ) = I_H (\nu) = I_{\mathcal L} (\nu),
            \end{equation}
            where we have used the fact that $ H(t) \le \mathcal L (t)  $ for $t\in [-1,1]$, so $I_{\mathcal L} (\mu ) \ge I_H (\mu )$.  Since $H$ is positive definite, according to Propositions \ref{prop:schon} and \ref{prop:pdmin}, we have that  $\sigma$ minimizes $I_H$, i.e. $I_H (\mu ) \ge I_H (\sigma)$. We have also used that the cross-polytope is a $3$-design and $H$ is a quadratic polynomial, hence $I_H (\sigma) = I_H (\nu)$. Finally, $H(t) = \mathcal L (t)$ for $t \in \{ x\cdot y: x,y \in \supp \nu \} = \{ 0, \pm 1\}$, hence $ I_H (\nu) = I_{\mathcal L} (\nu)$. This proves that the cross-polytope minimizes $I_{\mathcal L}$ for $\tau = \sqrt{2}$. 

            \subsubsection{Icosahedron}

            We shall need two facts about the icosahedron, namely that the set of inner products between elements of $\mathcal C$ is  $ \{ x\cdot y: x,y \in \mathcal C \} = \left\{ \pm 1, \pm 1/{\sqrt{5}} \right\}$, and that the icosahedron $\mathcal C$ is a $5$-design.  For simplicity let us consider the function $F(t) =  \frac{\mathcal L (t)}{\mathcal L (1)}$ so that $F(1) =1$ (obviously,  this does not effect the minimizers).

            We construct the following polynomial:
            \begin{align*}
                H(t) & = \frac{5  (5-\sqrt{5} )  }{32}  t^4   +  \frac58 t^3  +  \frac{3\sqrt{5} - 5}{16 }  t^2   -   \frac18  t + \frac{1-\sqrt{5}}{32} \\
                     & = \frac{5-\sqrt{5}}{28}   C_4 (t)  +  \frac14  C_3 (t)  +  \frac{20+3\sqrt{5}}{84}  C_2 (t)  +  \frac14  C_1 (t) +  \frac1{12}  C_0 (t), 
            \end{align*}
            where $C_k$ are the standard Legendre polynomials. We observe then that $H$ is positive definite, and that $H(t) \le F(t)$ for $-1\le t \le 1$, which follows from the formula $$H(t)=\frac{5}{32}(5-\sqrt{5})(t+1)(t-\frac{1}{\sqrt{5}})(t+\frac{1}{\sqrt{5}}).$$ A glance at this formula gives $H\leq F$ for $t\in[-1,\frac{1}{\sqrt{5}}]$, and the fact that $F-H$ is a polynomial with roots $$t=-1,\frac{1}{\sqrt{5}},\frac{-1\pm4\sqrt{10+4\sqrt{5}}}{\sqrt{5}},$$ gives $H\leq F$ for $t\in[\frac{1}{\sqrt{5}},1]$ which is a subset of the interval $[\frac{1}{\sqrt{5}},\frac{-1+4\sqrt{10+4\sqrt{5}}}{\sqrt{5}}]$. 

            $H$ coincides with $F$ (by construction) on the set $ \{ \langle x,y \rangle:  x,y \in \mathcal C \} = \left\{ \pm 1, \pm 1{\sqrt{5}} \right\}$. $H$ has a local maximum at $-\frac{1}{\sqrt{5}}$ and has been obtained by solving the linear equations  $H(t) = F(t) $ for $t =  \pm 1, \pm 1/{\sqrt{5}}$, as well as $H' (-1/\sqrt{5}) = 0$. The same argument as in the previous subsection finally shows
            $$\displaystyle{I_F (\nu) = \inf_{\mu \in \mathcal P}  I_F (\mu)},$$
            i.e.\ the icosahedron minimizes the energy $I_{\mathcal L}$ for $\tau^2 = \frac{2\sqrt{5}}{\sqrt{5} -1} $.
            \endgroup

%BibTex bibliography
% \bibliographystyle{abbrv} 
% \bibliography{refs}

\end{document}